\documentclass[a4paper,reqno,12pt]{amsart}
\usepackage[T1]{fontenc}
\usepackage[utf8]{inputenc}
\usepackage{amsmath,amssymb,amsfonts,amsthm,amscd}
\usepackage{mathrsfs}
\usepackage{calligra}
\usepackage{bm}
\usepackage{bbm}

\usepackage{booktabs}   
\usepackage{tabularx}   

\usepackage{tikz}
\usetikzlibrary{arrows.meta, positioning, calc}
\usepackage{pgf}
\usepackage{tikz-cd}
\usepackage{tkz-graph}
\usepackage[all,2cell]{xy}

\usepackage{graphicx}
\usepackage{multicol}
\usepackage{enumerate}
\usepackage{comment}
\usepackage{rotating}
\usepackage[normalem]{ulem}
\usepackage{xcolor}

\usetikzlibrary{decorations.markings,arrows}

\tikzset{
  midarrow/.style={
    postaction=decorate,
    decoration={
      markings,
      mark=at position .5 with {\arrow{stealth}}
    }
  }
}

\allowdisplaybreaks

\UseAllTwocells
\SilentMatrices

\numberwithin{equation}{section}

\usepackage[
  pagebackref,
  colorlinks=true,
  linkcolor=violet,
  citecolor=teal,
  urlcolor=blue,
  hypertexnames=false
]{hyperref}

\newtheorem{thm}{Theorem}[section]
\newtheorem{cor}[thm]{Corollary}
\newtheorem{lem}[thm]{Lemma}
\newtheorem{prop}[thm]{Proposition}

\theoremstyle{definition}
\newtheorem{dfn}[thm]{Definition}

\theoremstyle{remark}
\newtheorem{rmk}[thm]{Remark}

\newcommand{\rad}{\operatorname{rad}}
\newcommand{\nil}{\operatorname{Nil}}

\newcommand{\Prim}{\operatorname{Prim}}
\newcommand{\Spec}{\operatorname{Spec}}

\newcommand{\Aut}{\operatorname{Aut}}

\newcommand{\Int}{\operatorname{Int}}

\newcommand{\reg}{\operatorname{Reg}}

\newcommand{\gr}{\operatorname{gr}}

\def\-{\text{-}}

\newcommand{\Ann}{\operatorname{Ann}}

\newcommand{\KP}{\operatorname{KP}}

\newcommand{\N}{\mathbf{n}}
\newcommand{\M}{\mathbf{m}}
\newcommand{\bP}{\mathbf{p}}
\newcommand{\Q}{\mathbf{q}}

\begin{document}

\pagestyle{headings}

\title[Talented monoids of higher-rank graphs]{The prime spectrum of the talented monoid of a higher-rank graph and applications}

\author{Roozbeh Hazrat}
\address{Roozbeh Hazrat: Centre for Research in Mathematics and Data Science\\Western Sydney University, Australia} \email{\url{r.hazrat@westernsydney.edu.au}}

\author{Promit Mukherjee}
\address{Promit Mukherjee: Department of Mathematics, Jadavpur University, Kolkata-700032, India} \email{\url{promitmukherjeejumath@gmail.com}}



\begin{abstract}
In this paper, we further investigate the role of the graded Grothendieck group $K_0^{\gr}$ and its positive cone (the talented monoid) as an effective tool for distinguishing structural types of algebras associated to higher-rank $k$-graphs. We study the prime spectrum (the space of all prime $\Gamma$-order ideals equipped with a Zariski-like topology) of a general commutative $\Gamma$-monoid and establish that this space is spectral in the sense of Hochster, provided the monoid has the refinement property and every $\Gamma$-order ideal is finitely generated. As a result we are able to show that the prime spectrum of the talented monoid of a row-finite $k$-graph without sources and with a finite set of vertices is spectral. 

For any row-finite $k$-graph $\Lambda$ without sources, one of our main results says that the space of all graded prime ideals of the Kumjian--Pask algebra $\KP(\Lambda)$ is homeomorphic to both the space of all prime $\mathbb{Z}^k$-order ideals and the space of all prime $\mathbb{Z}^k$-filters of the talented monoid $T_\Lambda$. Another main result of this paper provides a complete topological description of regular $\Gamma$-order ideals of a refinement $\Gamma$-monoid: a $\Gamma$-order ideal $J$ is regular if and only if the corresponding closed (resp., open) set $V(J)$ (resp., $D(J)$) is regular closed (resp., regular open) in the prime spectrum. As an application of these results, we establish a lattice isomorphism between the lattice of all regular $\mathbb{Z}^k$-order ideals of the talented monoid and the lattice of all regular graded ideals of the Kumjian--Pask algebra via a spectrum-theoretic approach.  


\end{abstract}

\maketitle

\tableofcontents

\section{Introduction}\label{sec intro}

It was already recognized by Cuntz and Krieger~\cite{cuntzkrieger} in the early 1980s that purely infinite simple $C^*$-algebras arising from certain matrices are intimately connected with symbolic dynamics~\cite{LindMarcus}. Krieger’s dimension group~\cite{krieger1} provides a complete invariant for shift equivalence, establishing a fundamental connection between operator algebras and dynamical systems. With the introduction of Leavitt path algebras~\cite{lpabook}, this line of research moved towards a discrete algebraic setting, where the graded Grothendieck group, rather than ordinary \(K\)-theory alone, emerged as a natural bridge between algebra and dynamics \cite{haz2013,Hazrat-main}. 

Higher-rank graphs~\cite{KumjianPask} provide a higher-dimensional analogue of graphs and give rise to higher-dimensional Leavitt path algebras, commonly known as Kumjian--Pask algebras~\cite{Pino}. It has been established that these algebras could provide a natural and satisfactory algebraic framework for two-dimensional symbolic dynamics. However, an important question remains: \emph{what plays the role of Krieger’s dimension group in higher-dimensional dynamics?}

The graded Grothendieck group of higher-rank graph algebras, or equivalently its positive cone, the talented monoid, which can be defined directly from the higher-rank graph, appears to be a strong candidate for such an invariant. The talented monoid is proved to be quite effective in capturing important geometries of a higher-rank graph  (see \cite{HMPS1,HLM}), which govern the structural properties of the associated Kumjian--Pask algebra. Our program is to explore the triangle formed by higher-dimensional dynamics, higher-dimensional algebras, and graded Grothendieck groups, and to investigate how the talented monoid may serve as a junction connecting these three perspectives. An important part of this big program is to find an answer to the following question.  

\textbf{The guiding question:} \emph{Given row-finite higher-rank graphs $\Lambda$ and $\Omega$ without sources, if there is a $\mathbb Z^k$-monoid isomorphism between the talented monoids $T_\Lambda$ and $T_\Omega$, then to what extent can one conclude that the associated Kumjian--Pask algebras $\operatorname{KP}(\Lambda)$ and $\operatorname{KP}(\Omega)$ have the same structural type?} 


In this quest, so far we have obtained talented monoid criteria for some important classes of Kumjian--Pask algebras (see Table \ref{tab:kp_algebras_criteria} below; the criteria for primeness and pure infiniteness are obtained in this paper, while the rest are collected from \cite{HMPS1, HLM}). In cases where the criterion is both necessary and sufficient, talented monoids (equivalently, the graded Grothendieck group) can effectively distinguish the relevant class of Kumjian--Pask algebras.

\begin{table}[htbp]
    \centering
    \small 
    \begin{tabularx}{\textwidth}{
        >{\hsize=1.1\hsize\raggedright\arraybackslash}X 
        >{\hsize=1.4\hsize\raggedright\arraybackslash}X 
        >{\hsize=0.5\hsize\raggedright\arraybackslash}X
    }
        \toprule
        \textbf{Structural Type} & \textbf{Talented Monoid Criterion} & \textbf{Nature} \\
        \midrule
        \midrule
        Graded-simple & $T_\Lambda$ is simple as a $\mathbb{Z}^k$-monoid & Necessary and sufficient \\
        \addlinespace
        Simple & $T_\Lambda$ is simple as a $\mathbb{Z}^k$-monoid and $\mathbb{Z}^k$ acts freely on $T_\Lambda$ & Sufficient \\
        \addlinespace
        Semisimple & $T_\Lambda$ is atomic and $\mathbb{Z}^k$ acts freely on $T_\Lambda$ & Necessary and sufficient \\
        \addlinespace
        Purely infinite simple & $T_\Lambda$ is simple as a $\mathbb{Z}^k$-monoid and $\mathbb{Z}^k$ acts freely on $T_\Lambda$, and for any $0\neq b\in T_\Lambda$, there exist $a\in T_\Lambda$ and $\mathbf{n}\in \mathbb{N}^k$ such that $b\ge a > {}^{\mathbf{n}}a$ & Sufficient \\
        \addlinespace
        Ultramatricial ($\Lambda^0$ is finite) & $\mathbb{Z}^k$ does not act freely on $T_\Lambda$ & Necessary \\
        \addlinespace
        Locally finite ($\Lambda^0$ is finite) & For any $0\neq a\in T_\Lambda$ and $i=1,2,\ldots,k$, there exists $n_i\in \mathbb{N}\setminus \{0\}$ such that ${}^{n_i\mathbf{e}_i}a=a$ & Necessary and sufficient \\
        \addlinespace
        Crossed product ($\Lambda^0$ is finite) & ${}^{\mathbf{n}} \epsilon_\Lambda=\epsilon_\Lambda$ for all $\mathbf{n}\in \mathbb{Z}^k$ & Necessary and sufficient \\
        \addlinespace
        Prime & \{0\} is a prime $\mathbb{Z}^k$-order ideal of $T_\Lambda$ & Necessary and sufficient\\
        \addlinespace
        Purely infinite & For any $\mathbb{Z}^k$-order prime ideal $J$ of $T_\Lambda$, $\mathbb{Z}^k$ acts freely on the quotient $S=T_\Lambda/J$, and for any $0\neq a\in S$, there exist $x\in S$ and $\mathbf{n}\in \mathbb{N}^k$ such that ${}^{\mathbf{n}}x < x \le a$ & Sufficient \\
        \addlinespace
        \bottomrule
    \end{tabularx}
\vspace{0.15cm}
\caption{Talented monoid criteria for different structural types of Kumjian--Pask algebras}
    \label{tab:kp_algebras_criteria}
\end{table}

The objective of this paper is to explore whether the talented monoid of a higher-rank graph carries information about the (graded) prime and primitive ideals of higher-rank graph algebras.

It is evident from the mathematical literature that the study of prime ideals of rings and associative algebras helps substantially in understanding their ideal and module theory, and is therefore crucial to unfold the properties of these algebraic structures. The interesting fact is that the collection of prime ideals, called the \emph{prime spectrum}, forms a topological space under the so-called Zariski topology, which enables one to gather information about the ring (algebra) by studying the topological space and vice-versa. 

There is a significant amount of work that has been devoted to the study of prime and primitive ideals of different combinatorial algebras; see \cite{BPRS, Hong} for graph $C^*$-algebras, \cite{ABR, APS} for Leavitt path algebras, \cite{KangPask} for higher-rank graph $C^*$-algebras, and \cite{Larki} for Kumjian--Pask algebras. The common fact that emerges from these works is that there is a close relationship between the prime (and primitive) ideals of graph (higher-rank graph) algebras and certain subsets of vertices of the underlying graph (higher-rank graph), known as \emph{maximal tails}. As the talented monoid of a higher-rank graph is defined directly from the combinatorial data, we are naturally motivated to investigate whether the maximal tails of a higher-rank graph can be described using the monoid-theoretic language. This sets the stage for the development of this paper. We show that the maximal tails are in an order-reversing bijection with the prime $\mathbb{Z}^k$-order ideals of the talented monoid. Motivated by this, we execute in this paper a detailed study of the space of all prime $\Gamma$-order ideals of an arbitrary commutative $\Gamma$-monoid $M$. The space is denoted by $\Spec_\Gamma(M)$ and is equipped with a Zariski-like topology. We show that when the monoid $M$ has some nice properties, then $\Spec_\Gamma(M)$ behaves quite like the prime spectrum of a commutative unital ring; for example, as one of the major results, we are able to show that $\Spec_\Gamma(M)$ is spectral in the sense of Hochster \cite{Hochster}, when $M$ is conical, has the refinement property, and every $\Gamma$-order ideal is finitely generated. As a consequence, it follows that the prime spectrum of the talented monoid of a row-finite higher-rank graph without sources and with finitely many vertices is spectral (Corollary \ref{cor spec of TM is compact}). 

Using the algebraic preorder of a commutative $\Gamma$-monoid $M$, we also define prime $\Gamma$-filters on $M$ and show that these are exactly the duals of prime $\Gamma$-order ideals. This helps us to convert the anti-isomorphism between the poset of maximal tails of a higher-rank graph and the poset of all prime $\mathbb{Z}^k$-order ideals of the talented monoid into an order-isomorphism between the former and the poset of all prime $\mathbb{Z}^k$-filters of the talented monoid (Theorem \ref{th maximal tails and prime filters}). It is known that the space of maximal tails of a higher-rank graph $\Lambda$ can be given a specific $T_0$ topology, and when the $k$-graph is strongly aperiodic, this space is homeomorphic to the primitive ideal space (which coincides with the prime spectrum) of the associated $C^*$-algebra $C^*(\Lambda)$ as well as the Kumjian--Pask algebra $\KP(\Lambda)$ of $\Lambda$ equipped with the Jacobson hull-kernel topology (see \cite{KangPask} and \cite{Larki}). Here we show that the space of maximal tails is homeomorphic to the prime spectrum (and the space of prime $\mathbb{Z}^k$-filters) of the talented monoid (Theorem \ref{th the three term connection for KPA}). This, together with the spectrality of the prime spectrum of the talented monoid, shows that the space of maximal tails of $\Lambda$ and the space of graded prime ideals of $\KP(\Lambda)$ are spectral in the strong sense of Hochster, and provides a new proof of the sobriety of the space of maximal tails (\cite[Proposition 4.4]{KangPask}) from a talented monoid perspective. Moreover, as immediate consequences, we are able to show that the talented monoid can detect the class of prime Kumjian--Pask algebras (Theorem \ref{cor TM characterization of prime KPA}) and also the class of primitive higher-rank graph algebras (Theorem \ref{th TM detects primitive ideal space}). 

The study of regular ideals of graph algebras has seen a notable surge of activity in recent years (see for instance \cite{Pitts,RegLPA,Schenkel}). In \cite{CGHaz}, Cordeiro, Gonc\c alves and the first author defined the notions of orthogonality and regularity in the setting of $\Gamma$-monoids. Specializing to talented monoids of directed graphs, it was established in \cite[Corollary 4.5]{CGHaz} that there is a one-to-one correspondence between the regular $\mathbb{Z}$-order ideals of the talented monoid of a row-finite directed graph and the regular ideals of the corresponding Leavitt path algebra. As another important result of this paper, we prove Theorem \ref{th regular ideals via the prime spectrum}, which sets up an isomorphism between the poset of all regular $\Gamma$-order ideals of a commutative conical refinement $\Gamma$-monoid $M$ and the Boolean algebra of all regular open sets of $\Spec_\Gamma(M)$. In order to extend \cite[Corollary 4.5]{CGHaz} to higher dimensions, we here adopt a novel approach by using Theorem \ref{th regular ideals via the prime spectrum} and the established homeomorphism between the prime spectrum of the talented monoid and the graded prime spectrum of the Kumjian--Pask algebra. This yields Theorem \ref{th lattice isomorphism of regular ideals via TM} showing that the isomorphism between talented monoids of two higher-rank graphs leaves invariant the class of regular graded ideals of the associated Kumjian--Pask algebras. From this theorem, it follows that the homeomorphism between the prime spectra of the talented monoid and the Kumjian--Pask algebra preserves the respective subspaces of regular primes. We also identify the geometric counterpart of the subspace of regular prime $\mathbb{Z}^k$-order ideals as the subspace of those maximal tails $M$, which coincide with the \emph{reachability sets} (\ref{eq reachability}) of the \emph{inaccessibility sets} (\ref{eq inaccessibility}) of $\Lambda^0\setminus M$.   

We now briefly describe the arrangement of the paper. 

In Section \ref{sec prelim}, we recall the preliminary concepts, which include the relevant topological and algebraic concepts, a brief overview of higher-rank graphs, Kumjian--Pask algebras, and $\Gamma$-monoids. 

Section \ref{sec prime spec of monoid} offers a study of the prime spectrum $\Spec_\Gamma(M)$ of a commutative $\Gamma$-monoid $M$. In Proposition \ref{pro basic facts about basic open sets}, we prove several topological properties of this space equipped with a Zariski-like topology. In particular, we show that under certain conditions, this is a spectral space. This directly shows that the prime spectrum of the talented monoid of a row-finite $k$-graph without sources and with a finite object set is spectral (Corollary \ref{cor spec of TM is compact}). 

In Section \ref{sec regular ideals}, we focus on the regular $\Gamma$-order ideals of a conical refinement monoid $M$ and provide a spectrum-theoretic description of these ideals in Theorem \ref{th regular ideals via the prime spectrum}. As a consequence, we show that the poset of all regular $\Gamma$-order ideals becomes a complete Boolean algebra and give explicit descriptions of the join, meet and complement. In this section, we also introduce the notion of (prime) $\Gamma$-filters of a $\Gamma$-monoid $M$ and connect these with the prime $\Gamma$-order ideals (see Proposition \ref{pro dual of prime order ideals} and Corollary \ref{cor regular prime order ideals}). 

In Section \ref{sec maximal tail and TM}, we concentrate on talented monoids of higher-rank graph which are special combinatorial $\mathbb{Z}^k$-monoids. We first show that the maximal tails of a $k$-graph $\Lambda$ can effectively be described via the talented monoid (Theorem \ref{th maximal tails via TM} and Theorem \ref{th maximal tails and prime filters}), which eventually leads us to establish that the space of all maximal tails of $\Lambda$, the space of all prime $\mathbb{Z}^k$-order ideals of $T_\Lambda$, the space of all prime $\mathbb{Z}^k$-filters of $T_\Lambda$, and the space of all graded prime ideals of the Kumjian--Pask algebra $\KP(\Lambda)$ are all homeomorphic to each other (Theorem \ref{th the three term connection for KPA}). We also provide necessary and sufficient graph-theoretic and talented monoid-theoretic criteria for the graded prime spectrum of the Kumjian--Pask algebra to be connected as a topological space.

The paper concludes in Section \ref{sec applications}, where using the connection between spectra of monoids and algebras in the previous sections, we show that the Boolean algebra of regular $\mathbb{Z}^k$-order ideals of the talented monoid is isomorphic to the Boolean algebra of all regular graded ideals of the Kumjian--Pask algebra (Theorem \ref{th lattice isomorphism of regular ideals via TM}). Apart from this, we also obtain Theorem \ref{th sufficient criterion for purely infinite higher-rank algebras} which gives talented monoid criteria for purely infinite higher-rank graph algebras ($C^*$-algebras and Kumjian--Pask algebras).

\section{Background materials}\label{sec prelim}
\subsection{General concepts} Since in this paper we will be dealing with the graded prime spectrum of certain graded rings (Kumjian--Pask algebras and Leavitt path algebras), it is important to recall the prerequisite ring-theoretic and topological concepts.

We start with some topological preliminaries. Let $X$ be a topological space. A subset $U$ of $X$ is called \emph{regular open} (resp., \emph{regular closed}) if $U=\Int(\overline{U})$ (resp., $U=\overline{\Int(U)}$). The collection of all regular open (closed) sets forms a complete Boolean algebra partially ordered by set inclusion. For two regular open sets $U$ and $V$, their meet is defined as $U\wedge V:=U\cap V$, and their join is defined as $U\vee V:=\Int(\overline{U\cup V})$. The complement of a regular open set $U$ is given as $U':=\Int(X\setminus U)=X\setminus \overline{U}$. We denote the Boolean algebra of all regular open sets of $X$ by $\mathcal{R}\mathcal{O}(X)$. For a more comprehensive idea about the notion of regular open sets, the reader may consult \cite[Chapter 9--10]{BooleanBook}.

We now recall the definition of spectral spaces.

\begin{dfn}\label{def spectral spaces}
Let $X$ be a topological space. Then $X$ is called

$(i)$ \emph{spectral in the sense of Hofmann--Keimel} \cite{HofKei} if $X$ is $T_0$ and every nonempty irreducible set has a generic point.

$(ii)$ \emph{spectral in the sense of Hochster} \cite{Hochster} if $X$ is compact, admits a base of compact sets closed under finite intersection, and is spectral in the sense of Hofmann--Keimel.
\end{dfn}

Suppose $R$ is a (not necessarily unital) ring (or algebra). 

\begin{itemize}
    \item[$(i)$] A proper ideal $P$ of $R$ is called \emph{prime} if for any two ideals $I,J$ of $R$, $IJ\subseteq P$ implies either $I\subseteq P$ or $J\subseteq P$. The ring $R$ is called a \emph{prime ring} if $\{0\}$ is a prime ideal.

    \item[$(ii)$] A proper ideal $I$ of $R$ is called \emph{left primitive} (resp., \emph{right primitive}) if $I=\Ann_R(M)$ for some simple left (resp., right) $R$-module $M$. The ring $R$ is called \emph{left primitive} (resp., \emph{right primitive}) if $\{0\}$ is a left (resp., right) primitive ideal. If $R$ is both left and right primitive, then it is called \emph{primitive}. 
\end{itemize}

The set of all prime (resp., primitive) ideals of a ring $R$ is denoted by $\Spec(R)$ (resp., $\Prim(R)$). A standard fact in ring theory is that $\Prim(R)\subseteq \Spec(R)$. The set $\Spec(R)$ is equipped with the well-known Zariski topology and is called the \emph{prime spectrum} of $R$. In the remarkable paper of Hochster \cite{Hochster}, it was established that a topological space $X$ is spectral (according to Definition \ref{def spectral spaces}$(ii)$) if and only if $X$ is homeomorphic to $\Spec(R)$ for a commutative unital ring $R$. 

Let $R$ be a ring and $\Gamma$ an abelian group. Then $R$ is called \emph{graded by $\Gamma$} or a $\Gamma$-\emph{graded ring} if there is a collection $\{R_\alpha~|~\alpha\in \Gamma\}$ of additive subgroups of $R$ such that 

$(i)$ $R=\displaystyle{\bigoplus_{\alpha\in \Gamma}}~ R_\alpha$ as groups, and 

$(ii)$ $R_\alpha R_\beta\subseteq R_{\alpha\beta}$. 

The subgroup $R_\alpha$ is called the \emph{$\alpha^{\text{th}}$ homogeneous component} of $R$. The \emph{set of homogeneous elements} of $R$ is defined as \[R^h:=\displaystyle{\bigcup_{\alpha\in \Gamma}} R_\alpha.\] An ideal $I$ of $R$ is said to be a \emph{graded ideal} if $I=\displaystyle{\bigoplus_{\alpha\in \Gamma}}~(I\cap R_\alpha)$. 

When $R$ is a $\Gamma$-graded ring, one can define the graded analogue of the notion of prime ideals. A proper graded ideal $P$ of $R$ is called \emph{graded prime} if for any two graded ideals $I,J$ of $R$, $IJ\subseteq P$ implies $I\subseteq P$ or $J\subseteq P$. The set of all graded prime ideals of $R$ is denoted by $\Spec^{\gr}(R)$ and is called the \emph{graded spectrum} of $R$. 

One can define a certain topology on $\Spec_\Gamma(R)$ which mimics the Zariski topology on the spectrum of a commutative ring. For any graded ideal $I$ of $R$, we denote \[V^{\gr}(I):=\{P\in \Spec^{\gr}(R)~|~I\subseteq P\}\] and \[D^{\gr}(I):=\Spec^{\gr}(R)\setminus V^{\gr}(I).\] It is then easy to check that the sets $V^{\gr}(I)$ satisfy the axioms to be the closed sets of a topology on $\Spec^{\gr}(R)$. A moment's thought yields that this topology has a basis consisting of open sets of the form \[D^{\gr}(x):=\{P\in \Spec^{\gr}(R)~|~x\notin P\},\] where $x\in R^h$.   

We finish this subsection by recalling purely infinite rings.

Let $R$ be any ring. For $a,b\in R$, we write $b\precsim a$ if there exist $x,y\in R$ such that $b=xay$. The ring $R$ is called \emph{purely infinite} \cite{AGPM} if 

$(i)$ no quotient of $R$ is a division ring, and 

$(ii)$ for all $a,b\in R$ with $b\in RaR$, $b\precsim a$.

If $\mathcal{A}$ is a $C^*$-algebra which is purely infinite according to the above definition, then it is purely infinite in the sense of Kirchberg--R\o rdam (\cite[Proposition 3.17]{AGPM}). 

\subsection{Higher-rank graphs and Kumjian--Pask algebras}\label{ssec hr graph} 
We record the definition of higher-rank graphs and the basic notions that will be used in this paper. For the concepts that are not described here, we refer the reader to \cite{Pino,HLM, HMPS1, KumjianPask}.

Throughout the paper, $\mathbb{N}$ stands for the set of nonnegative integers. 
\begin{dfn}\label{def higher-rank graphs}
Let $k$ be any positive integer. A \emph{higher-rank $k$-graph} is a pair $(\Lambda,d)$, where $\Lambda$ is a countably small category and $d:\Lambda\longrightarrow \mathbb{N}^k$ is a functor (called the \emph{degree} functor) which satisfies the \emph{factorization} property: whenever $d(\lambda)=\M+\N$ for some $\M,\N\in \mathbb{N}^k$, then there are unique $\alpha,\beta\in \Lambda$ such that $d(\alpha)=\M$, $d(\beta)=\N$, and $\lambda=\alpha\beta$. When the degree functor $d$ is understood from the context, the $k$-graph is simply denoted as $\Lambda$. 
\end{dfn}
The domain and codomain of $\lambda\in \Lambda$ are denoted by $s(\lambda)$ and $r(\lambda)$, respectively, and are called the \emph{source} and \emph{range} of $\lambda$. For $u,v\in \Lambda^0$ and $\N\in \mathbb{N}^k$, \[\Lambda^\N:=d^{-1}(\N),~u\Lambda^\mathbf{n}:=r_\Lambda^{-1}(u)\cap \Lambda^\mathbf{n},~\Lambda^\mathbf{n} v:=\Lambda^\mathbf{n}\cap s_\Lambda^{-1}(v),~u\Lambda^\mathbf{n} v:=u\Lambda^\mathbf{n} \cap \Lambda^\mathbf{n} v.\] 
Using the factorization property, it is easy to observe that the set of objects of the category $\Lambda$ coincides with $\Lambda^0$. The elements of $\Lambda^0$ are often called \emph{vertices}, and any element of $\Lambda$ is called a \emph{path}. 

For $\lambda\in \Lambda$, and $0\le \bP\le \Q\le d(\lambda)$, $\lambda(\bP,\Q)$ denotes the unique path (which exists by the factorization property) such that $\lambda=\alpha \lambda(\bP,\Q)\beta$ for some $\alpha\in \Lambda^{\bP}$ and $\beta\in \Lambda^{d(\lambda)-\Q}$.

A $k$-graph $\Lambda$ is called \emph{row-finite} if $|u\Lambda^\N|<\infty$ for all $u\in \Lambda^0$ and $\N\in \mathbb{N}^k$, and is called \emph{without sources} if $u\Lambda^\N\neq \emptyset$ for all $u\in \Lambda^0$ and $\N\in \mathbb{N}^k$. 

Certain subsets of the set of vertices $\Lambda^0$ will play important roles throughout the paper. We now collect the definitions.
\begin{dfn}\label{def H and S subsets}
Let $\Lambda$ be a row-finite $k$-graph without sources. A subset $H$ of $\Lambda^0$ is called 

$(i)$ \emph{hereditary}, if $v\in H$ implies $s(v\Lambda^\N)\subseteq H$ for all $\N\in \mathbb{N}^k$;

$(ii)$ \emph{saturated}, if $s(v\Lambda^\N)\subseteq H$ for some $\N\in \mathbb{N}^k$ implies $v\in H$. 
\end{dfn}
For any $X\subseteq \Lambda^0$, the smallest hereditary saturated subset containing $X$ is denoted by $\overline{X}$. The collection of all the hereditary saturated subsets of $\Lambda^0$ forms a lattice denoted by $\mathcal{H}_\Lambda$, where the join and meet operations are respectively defined as \[H_1\vee H_2:=\overline{H_1\cup H_2},~~H_1\wedge H_2:=H_1\cap H_2.\]
\begin{dfn}\label{def maximal tail}
Let $\Lambda$ be a row-finite $k$-graph without sources. A nonempty subset $M\subseteq \Lambda^0$ is called a \emph{maximal tail} if the following conditions are satisfied:

(MT1) whenever $w\in M$ and $v\in \Lambda^0$ with $v\Lambda w\neq \emptyset$, then $v\in M$;

(MT2) for any $v\in M$ and $i\in \{1,2,\ldots,k\}$, there exists $\lambda\in v\Lambda^{\mathbf{e}_i}$ such that $s(\lambda)\in M$;

(MT3) for any two $u,v\in M$, there exists $w\in M$ such that $u\Lambda w, v\Lambda w\neq \emptyset$. 
\end{dfn}

It is easy to observe that $H\subseteq \Lambda^0$ is a hereditary saturated subset if and only if $\Lambda^0\setminus H$ satisfies Condition (MT1) and (MT2) above. Accordingly, for any maximal tail $M$, the complement $\Lambda^0\setminus M$ is a hereditary saturated subset.

In \cite{KumjianPask}, Kumjian and Pask defined a universal $C^*$-algebra $C^*(\Lambda)$ corresponding to a row-finite $k$-graph $\Lambda$ without sources. This generalizes the Cuntz--Krieger algebra of a row-finite directed graph. In 2013, Aranda Pino et al. \cite{Pino} defined purely algebraic analogues of $k$-graph $C^*$-algebras which are known as Kumjian--Pask algebras. As the substitute of adjoints of elements in $C^*(\Lambda)$, they used \emph{ghost paths} in the discrete set-up. For any $\lambda\in \Lambda\setminus \Lambda^0$, the ghost path corresponding to $\lambda$ is denoted by $\lambda^*$ and is nothing but a formal symbol. 
\begin{dfn}\label{def KP algebra}
Let $\Lambda$ be a row-finite $k$-graph without sources. The \emph{Kumjian--Pask algebra} of $\Lambda$ over a field $\mathsf{F}$ is denoted by $\KP_\mathsf{F}(\Lambda)$ and is defined to be the universal associative $\mathsf{F}$-algebra generated by formal symbols $p_v,s_\lambda,s_{\lambda^*}$; $v\in \Lambda^0,\lambda\in \Lambda\setminus \Lambda^0$ subject to the following relations:

(KP1) $p_up_v=\delta_{u,v}p_u$ for all $u,v\in \Lambda^0$;

(KP2) $s_\lambda s_\mu=s_{\lambda\mu},$ $s_{\mu^*}s_{\lambda^*}=s_{(\lambda\mu)^*}$ for all $\lambda,\mu\in \Lambda^{\neq 0}$ with $s(\lambda)=r(\mu)$, and

$p_{r(\lambda)}s_\lambda=s_\lambda=s_\lambda p_{s(\lambda)},$ $p_{s(\lambda)}s_{\lambda^*}=s_{\lambda^*}=s_{\lambda^*}p_{r(\lambda)}$ for all $\lambda\in \Lambda^{\neq 0}$;

(KP3) $s_{\lambda^*} s_\mu=\delta_{\lambda,\mu}p_{s(\lambda)}$ for all $\lambda,\mu\in \Lambda^{\neq 0}$ with $d(\lambda)=d(\mu)$;

(KP4) $p_v=\displaystyle{\sum_{\lambda\in v\Lambda^\N}} s_\lambda s_{\lambda^*}$ for all $v\in \Lambda^0$ and for all $\N\in \mathbb{N}^k\setminus \{0\}$.
\end{dfn}

In this paper, for any $k$-graph $\Lambda$, $\KP(\Lambda)$ will denote the Kumjian--Pask algebra of $\Lambda$ over an arbitrary field $\mathsf{F}$. When $\Lambda=E^*$, the free category of a directed graph $E$, then $\KP(\Lambda)$ becomes the well-known Leavitt path algebra $L(E)$ of $E$. In this way, Kumjian--Pask algebras generalize Leavitt path algebras in higher dimensions. 

\subsection{$\Gamma$-monoids} 
Throughout this paper, all monoids are assumed to be commutative. 

Recall that any monoid $M$ is equipped with a natural preorder: \[x\le y\Leftrightarrow y=x+z~\text{for some}~z\in M.\] This is usually called the \emph{algebraic preorder} on $M$. For $x,y\in M$, we write $x~||~y$ if $x$ and $y$ are incomparable, i.e., if $x\nleq y$ and $y\nleq x$. If the monoid is conical and cancellative, then the algebraic preorder becomes a partial order. 

A monoid $M$ is said to have the \emph{refinement} property (or sometimes $M$ is a \emph{refinement} monoid) if whenever $a+b=c+d$ for some $a,b,c,d\in M$, then there are $z_1,z_2,z_3,z_4\in M$ such that \[a=z_1+z_2,~b=z_3+z_4,~c=z_1+z_3,~d=z_2+z_4.\]
We now recall the definition of a $\Gamma$-monoid. 
\begin{dfn}\label{def Gamma-monoid}
Let $\Gamma$ be a group. A monoid $M$ is called a \emph{$\Gamma$-monoid} if there is an action of $\Gamma$ on $M$ given by monoid automorphisms. More precisely, $M$ is a $\Gamma$-monoid if there is a group homomorphism $\eta:\Gamma\longrightarrow \Aut(M)$. For any $x\in M$ and $\gamma\in \Gamma$, $^\gamma x$ denotes $\eta(\gamma)(x)$.    
\end{dfn}
When $M$ is a $\Gamma$-monoid, the action of $\Gamma$ on $M$ naturally induces an action of the group semiring $\mathbb{N}[\Gamma]$ on $M$ as follows: \[^{\left(\displaystyle{\sum_{i=1}^{t}}~n_i\gamma_i\right)}x:=\displaystyle{\sum_{i=1}^{t}}~n_i~^{\gamma_i}x.\] for all $\displaystyle{\sum_{i=1}^{t}}~n_i\gamma_i\in \mathbb{N}[\Gamma]$. An element $\epsilon\in M$ is called an \emph{order unit} if for any $x\in M$, there exists $\alpha\in \mathbb{N}[\Gamma]$ such that $x\le~^\alpha \epsilon$. 

Let $\Gamma$ be a group and $M$ a $\Gamma$-monoid. A $\Gamma$-\emph{order ideal} of $M$ is defined as a submonoid $J$, which is a down-set with respect to the algebraic preorder on $M$ and is also closed under the action of $\Gamma$. 

For a subset $X$ of $M$, we denote by $\langle X\rangle$ the $\Gamma$-order ideal generated by $X$. To be precise, if $X\neq \emptyset$, then \[\langle X\rangle:=\left\{y\in M~|~y\le \displaystyle{\sum_{i=1}^{t}}~^{\alpha_i}x_i~\text{for some}~t\in \mathbb{N},\alpha_i\in \mathbb{N}[\Gamma]~\text{and}~x_i\in X\right\}\] and $\langle \emptyset\rangle:=\{0\}$. For any element $x\in M$, we write $\langle x \rangle$ to denote $\langle \{x\}\rangle$. 

The intersection of two $\Gamma$-order ideals $J_1,J_2$ of $M$ is again a $\Gamma$-order ideal of $M$. However, the union may not be a $\Gamma$-order ideal. The collection of all $\Gamma$-order ideals becomes a lattice under inclusion where the meet is the intersection and the join is the $\Gamma$-order ideal generated by the union. When $M$ is a refinement monoid, it can be shown that $\langle J_1\cup J_2\rangle=J_1+J_2$, where \[J_1+J_2:=\{x+y~|~x\in J_1~\text{and}~y\in J_2\}.\] 

A $\Gamma$-monoid $M$ is called \emph{decomposable} if there exist proper $\Gamma$-order ideals $J_1,J_2$ of $M$ such that $J_1\cap J_2=\{0\}$ and $\langle J_1\cup J_2\rangle=M$. If no such decomposition exists, then $M$ is called \emph{indecomposable}.    

Let $J$ be any $\Gamma$-order ideal of a $\Gamma$-monoid $M$. Consider the relation $\kappa_J$ on $M$ defined as follows: $(x,y)\in \kappa_J$ if and only if $x+a=y+b$ for some $a,b\in J$. Then it is easy to check that $\kappa_J$ is a congruence on $M$ (in fact, it is enough for $J$ to be a submonoid for $\kappa_J$ to be a congruence). We denote the quotient monoid $M/{\kappa_J}$ simply as $M/{J}$. This is also a $\Gamma$-monoid with respect to the induced action: $^\gamma[x]_J:=[^\gamma x]_J$, where $[x]_J$ denotes the equivalence class of $x\in M$ under the relation $\kappa_J$.

Let us finish this section with a special combinatorial $\Gamma$-monoid---the talented monoid of a higher-rank graph. 

\begin{dfn}\label{def talented monoid}
Let $\Lambda$ be a row-finite $k$-graph without sources. The \emph{talented monoid} of $\Lambda$ is denoted by $T_\Lambda$ and is defined as the free commutative monoid generated by the formal symbols $v(\N)$ (where $v\in \Lambda^0, \N\in \mathbb{Z}^k$) subject to the following relations:
\begin{equation}\label{TM equation}
    v(\N)=\displaystyle{\sum_{\alpha\in v\Lambda^{\M}}} s(\alpha)(\N+\M)
\end{equation}
for each $v\in \Lambda^0$, $\N\in \mathbb{Z}^k$ and $\M\in \mathbb{N}^k$.
\end{dfn}

The monoid $T_\Lambda$ is a conical refinement monoid (see \cite[Theorem 3.18]{HMPS1}) and is equipped with a natural action of the group $\mathbb{Z}^k$, defined on the generators as follows: $^\Q v(\N):=v(\N+\Q)$. This makes $T_\Lambda$ a $\mathbb{Z}^k$-monoid. If $|\Lambda^0|<\infty$, then $\epsilon_\Lambda:=\displaystyle{\sum_{v\in \Lambda^0}} v(0)$ is an order-unit of $T_\Lambda$. 

\section{The prime spectrum of a \texorpdfstring{$\Gamma$}{Gamma}-monoid}\label{sec prime spec of monoid}
We start with the definitions of some special types of $\Gamma$-order ideals of a $\Gamma$-monoid including prime $\Gamma$-order ideals, which are going to be our main objects of study. 

\begin{dfn}\label{def prime, primitive and maximal}
Let $\Gamma$ be a group and $M$ a $\Gamma$-monoid. Let $J$ be a proper $\Gamma$-order ideal. Then $J$ is called 

$(i)$ \emph{prime}, if for any two $\Gamma$-order ideals $J_1,J_2$ of $M$, $J_1\cap J_2\subseteq J$ implies $J_1\subseteq J$ or $J_2\subseteq J$,

$(ii)$ \emph{primitive}, if $J$ is prime and $\Gamma$ acts freely on the quotient $\Gamma$-monoid $M/{J}$, and 

$(iii)$ \emph{maximal}, if for any $\Gamma$-order ideal $I$ of $M$, $J\subseteq I\subseteq M$ implies $J=I$ or $I=M$. 
\end{dfn}   

The collection of all prime $\Gamma$-order ideals of $M$ is denoted by $\Spec_\Gamma(M)$ and is called the \emph{prime spectrum} of $M$. The intersection of all prime $\Gamma$-order ideals of $M$ is denoted by $\nil(M)$ and is called the \emph{nilradical} of $M$. 

Parallel to the prime spectrum of a commutative ring, one can impose a Zariski-like topology on $\Spec_\Gamma(M)$. For any $\Gamma$-order ideal $J$ of $M$, we write \[V(J):=\{P\in \Spec_\Gamma(M)~|~J\subseteq P\}.\] One can check that 

\begin{itemize}
\item $\emptyset=V(M)$, $\Spec_\Gamma(M)=V(\{0\})$;

\item $\displaystyle{\bigcup_{i=1}^{n}}~V(J_i)=V\left(\displaystyle{\bigcap_{i=1}^{n}}J_i\right)$;

\item $\displaystyle{\bigcap_{\nu\in \Delta}} V(J_\nu)=V\left(\langle \displaystyle{\bigcup_{\nu\in \Delta}}J_\nu\rangle \right)$ for any index set $\Delta$.
\end{itemize}

Hence the sets $V(J)$ satisfy the axioms to be the closed sets of a topology on $\Spec_\Gamma(M)$, which we call the \emph{Zariski topology} on $\Spec_\Gamma(M)$. One can also trace this topology by identifying the open sets. For each $x\in M$, let \[D(x):=\{P\in \Spec_\Gamma(M)~|~x\notin P\}.\] Then $D(x)=\Spec_\Gamma(M)\setminus V(\langle x\rangle)$, and so $D(x)$ is open in the Zariski topology. Moreover, for any $\Gamma$-order ideal $J$ of $M$, $\Spec_\Gamma(M)\setminus V(J)=\displaystyle{\bigcup_{x\in J}}~D(x)$ which shows that the collection $\{D(x)~|~x\in M\}$ forms a basis for the Zariski topology on $\Spec_\Gamma(M)$. Defining \[D(J):=\{P\in \Spec_\Gamma(M)~|~J\nsubseteq P\},\] it follows that $D(J)=\displaystyle{\bigcup_{x\in J}}~D(x)$. 

For a $\Gamma$-order ideal $J$ of a $\Gamma$-monoid $M$, the \emph{radical} of $J$ is denoted by $\rad(J)$ and is defined as $\rad(J):=\displaystyle{\bigcap_{P\in V(J)}}P$. Note that $\rad(\{0\})=\nil(M)$. 

We denote the collection of all primitive $\Gamma$-order ideals by $\Prim_\Gamma(M)$. We regard this as a topological space with respect to the subspace topology inherited from the Zariski topology of $\Spec_\Gamma(M)$. 

\begin{lem}\label{lem maximal order ideal is prime}
Let $\Gamma$ be a group and $M$ a $\Gamma$-monoid with the refinement property. Then every maximal $\Gamma$-order ideal of $M$ is prime.    
\end{lem}
\begin{proof}
Let $J$ be a maximal $\Gamma$-order ideal of $M$. Suppose $J_1,J_2$ are $\Gamma$-order ideals with $J_1\cap J_2\subseteq J$. Suppose $J_1\nsubseteq J$. We will show that $J_2\subseteq J$. Choose any $x\in J_1\setminus J$ and let $I:=\langle J\cup \{x\}\rangle$. Then $I$ is a $\Gamma$-order ideal of $M$ and $J\subsetneq I$ since $x\in I\setminus J$. By maximality of $J$, $I=M$. Let $y\in J_2$ be an arbitrary element. Then $y\in I$. So there are $a\in J$ and $\gamma\in \mathbb{N}[\Gamma]$ such that $y\le a+~^\gamma x$. Then we can write $a+~^\gamma x=y+z$ for some $z\in M$. Since $M$ is a refinement monoid, there exist $z_1,z_2,z_3,z_4$ such that $a=z_1+z_2$, $^\gamma x=z_3+z_4$ and $y=z_1+z_3$, $z=z_2+z_4$. Now $z_1\le a\in J$, which implies $z_1\in J$ since $J$ is an order ideal. Since $J_2$ is an order ideal and $z_3\le y\in J_2$, $z_3\in J_2$. Again since $z_3\le~^\gamma x$, $x\in J_1$ and $J_1$ is a $\Gamma$-order ideal, it follows that $z_3\in J_1$. Hence, $z_3\in J_1\cap J_2\subseteq J$ and consequently, $y=z_1+z_3\in J$ since $J$ is a submonoid. Since $y$ was chosen arbitrarily from $J_2$, this shows that $J_2\subseteq J$ and thus finishing the proof. 
\end{proof}

Recall that an element $\epsilon$ in a $\Gamma$-monoid $M$ is called an \emph{order unit}
if for any $x\in M$, there exists $\alpha\in \mathbb{N}[\Gamma]$ such that $x\le~^\alpha \epsilon$. Suppose $M$ is a $\Gamma$-monoid containing an order unit $\epsilon$. Let $J$ be any proper $\Gamma$-order ideal of $M$. Suppose $\mathscr{F}_J$ is the collection of all proper $\Gamma$-order ideals of $M$ containing $J$. Then $\mathscr{F}_J$ is partially ordered by inclusion. Let $C=\{J_\nu~|~\nu\in \Delta\}$ be a chain in $\mathscr{F}_J$. Let $I=\displaystyle{\bigcup_{\nu\in \Delta}} J_\nu$. If $\epsilon\in I$, then $\epsilon\in J_\nu$ for some $\nu\in \Delta$. But that would imply $x\in J_{\nu}$ for all $x\in M$ and hence $J_{\nu}=M$ since $\epsilon$ is an order unit. Therefore, $\epsilon\notin I$ and so $I\in \mathscr{F}_J$ and $J_\nu\subseteq I$ for all $\nu\in \Delta$. Thus, we can apply Zorn's lemma to conclude that $\mathscr{F}_J$ has a maximal element. Therefore, every proper $\Gamma$-order ideal is contained in a maximal $\Gamma$-order ideal, in particular, any nonzero conical $\Gamma$-monoid with an order unit contains a maximal $\Gamma$-order ideal, which is also prime if the monoid is refinement.

Although we need order units to guarantee the existence of maximal $\Gamma$-order ideals, the next lemma says that we do not need them for the existence of prime $\Gamma$-order ideals. 
\begin{lem}\label{lem proper ideal contained in a prime ideal}
Let $\Gamma$ be a group and $M$ a $\Gamma$-monoid with the refinement property. Then for every proper $\Gamma$-order ideal $J$ of $M$ and $x\notin J$, there is a prime $\Gamma$-order ideal $P$ such that $J\subseteq P$ and $x\notin P$. In particular, if $M$ is conical and $M\neq \{0\}$, then $\Spec_\Gamma(M)\neq \emptyset$.    
\end{lem}
\begin{proof}
Let $J$ be any proper $\Gamma$-order ideal of $M$ and $x\notin J$. We consider the collection \[\mathfrak{J}:=\{I~|~I~\text{is a}~\Gamma\text{-order ideal such that}~J\subseteq I~\text{and}~x\notin I\}.\] Since $J\in \mathfrak{J}$, $\mathfrak{J}$ is a nonempty partially ordered set ordered by inclusion. If $C=\{I_\nu~|~\nu\in \Delta\}$ is a chain in $\mathfrak{J}$, then it is clear that $\displaystyle{\bigcup_{\nu\in \Delta}}I_\nu$ is an upper bound for $C$. By Zorn's lemma, we have a maximal element $I$ in $\mathfrak{J}$. We now show that $I$ is a prime $\Gamma$-order ideal of $M$. If possible, suppose there are $\Gamma$-order ideals $J_1,J_2$ such that $J_1\cap J_2\subseteq I$ but $J_1\nsubseteq I$ and $J_2\nsubseteq I$. For $i=1,2$, pick elements $y_i\in J_i\setminus I$ and consider the $\Gamma$-order ideals $I+\langle y_i\rangle$. Since both these ideals contain $I$ properly, by maximality of $I$ in $\mathfrak{J}$, we should have $x\in I+\langle y_i\rangle$ for $i=1,2$. Then there exist $a,b\in I$, $u_i\in \langle y_i\rangle$, $i=1,2$ such that $x=a+u_1=b+u_2$. Applying the refinement property, we have $w_1,w_2,w_3,w_4\in M$ such that \[a=w_1+w_2,~~u_1=w_3+w_4,~~b=w_1+w_3,~~u_2=w_2+w_4.\] Then for $i=1,2$, $w_4\le u_i\in \langle y_i\rangle\subseteq J_i$, whence $w_4\in J_1\cap J_2\subseteq I$. Also since $w_3\le b\in I$, $w_3\in I$. These imply that $x=a+u_1=a+w_3+w_4\in I$. But this is a contradiction since $I\in \mathfrak{J}$. Therefore, $I$ is a prime $\Gamma$-order ideal of $M$ containing $J$ and not containing $x$. The remaining statement is evident by considering the trivial $\Gamma$-order ideal $\{0\}$. 
\end{proof}
We now record an interesting consequence of Lemma \ref{lem proper ideal contained in a prime ideal}.
\begin{prop}\label{pro free action on prime quotient is enough}
Let $\Gamma$ be a group, $M$ a $\Gamma$-monoid with the refinement property. Then the following are equivalent.

$(i)$ Every prime $\Gamma$-order ideal is primitive, i.e., $\Gamma$ acts freely on the quotient $M/{P}$ for any $P\in \Spec_\Gamma(M)$.

$(ii)$ $\Gamma$ acts freely on $M/{J}$ for any proper $\Gamma$-order ideal $J$ of $M$.
\end{prop}
\begin{proof}
$(ii)\Rightarrow (i)$ This is obvious.

$(i)\Rightarrow (ii)$ Let $J$ be any proper $\Gamma$-order ideal of $M$. Suppose $[0]_J\neq [x]_J\in M/{J}$ and $\gamma\in \Gamma$ are such that $^\gamma [x]_J=[x]_J$. We show that $\gamma=e_\Gamma$, where $e_\Gamma$ is the identity of $\Gamma$. Since $[x]_J\neq [0]_J$, $x\notin J$ and so by Lemma \ref{lem proper ideal contained in a prime ideal}, there is a prime $\Gamma$-order ideal $P$ of $M$ such that $J\subseteq P$ and $x\notin P$. Now $^\gamma [x]_J=[x]_J$ in $M/{J}$ implies $^\gamma x+a=x+b$ in $M$ for some $a,b\in J$. Then $a,b\in P$ and we have $^\gamma [x]_P=[x]_P$ in $M/{P}$. If $[x]_P=[0]_P$ in $M/{P}$, then $x+y=z$ for some $y,z\in P$. But this implies that $x\in P$ since $P$ is an order ideal. Thus, $[x]_P\neq [0]_P$ in $M/{P}$. Since $P$ is primitive, $\Gamma$ acts freely on $M/{P}$, whence $\gamma=e_\Gamma$. Hence the result. 
\end{proof}

We now derive the following list of results, most of which are well-known for the prime spectrum of a commutative ring with identity. This shows that prime $\Gamma$-order ideals in a \emph{nice} $\Gamma$-monoid behave much like the prime ideals of a commutative unital ring. 

\begin{prop}\label{pro basic facts about basic open sets}
Let $\Gamma$ be a group and $M$ a conical $\Gamma$-monoid such that $M\neq \{0\}$. Then the following hold.

$(i)$ $D(x)=\emptyset$ if and only if $x\in \nil(M)$.

$(ii)$ $D(x)\cap D(y)=D(\langle x\rangle \cap \langle y\rangle)$ for $x,y\in M$.

$(iii)$ $\overline{\{P\}}=V(P)$ for any $P\in \Spec_\Gamma(M)$.

$(iv)$ $\Spec_\Gamma(M)$ is $T_0$.

$(v)$ $V(J)=V(\rad(J))$ for any $\Gamma$-order ideal $J$ of $M$.

$(vi)$ For any $\Gamma$-order ideal $J$, $\rad(J)$ is prime if and only if $V(J)$ is a nonempty irreducible subset.

$(vii)$ $\Spec_\Gamma(M)$ is sober, i.e., every nonempty irreducible set has a unique generic point.

$(viii)$ $\Spec_\Gamma(M)$ is spectral in the sense of Hofmann--Keimel.

Furthermore, if $M$ is a refinement monoid, then 

$(1)$ $D(x)=\Spec_\Gamma(M)$ if and only if $x$ is an order unit in $M$;

$(2)$ for any $P\in \Spec_\Gamma(M)$, $\{P\}$ is closed if and only if $P$ is maximal;

$(3)$ $J=\rad(J)$ for any $\Gamma$-order ideal of $M$, in particular, $\nil(M)=\{0\}$;

$(4)$ $D(x)=\emptyset$ if and only if $x=0$;

$(5)$ $D(x)$ is compact for all $x\in M$, in particular, $\Spec_\Gamma(M)$ is compact. 

If, in addition, each $\Gamma$-order ideal of $M$ is finitely generated, then $\Spec_\Gamma(M)$ is spectral in the sense of Hochster. 
\end{prop}
\begin{proof}
$(i)$ This is obvious from the definition of nilradical of $M$. 

$(ii)$ This follows from the definition of a prime $\Gamma$-order ideal.

$(iii)$ Obviously, $V(P)$ is a closed set containing $P$. Now if $V(J)$ is any closed set containing $P$, then $J\subseteq P$, whence $V(P)\subseteq V(J)$. Thus, $V(P)=\overline{\{P\}}$.

$(iv)$ Let $P,Q\in \Spec_\Gamma(M)$ be such that $P\neq Q$. Then $V(P)\neq V(Q)$ which implies $\overline{\{P\}}\neq \overline{\{Q\}}$ by $(iii)$. So $\Spec_\Gamma(M)$ is $T_0$. 

$(v)$ If $P\in V(J)$, then $\rad(J)=\displaystyle{\bigcap_{Q\in V(J)}}Q\subseteq P$. Hence $P\in V(\rad(J))$. Since $J\subseteq \rad(J)$, the other inclusion is obvious. 

$(vi)$ Assume that $V(J)$ is a nonempty irreducible subset. Suppose $J_1$, $J_2$ are $\Gamma$-order ideals such that $J_1\cap J_2\subseteq \rad(J)$. Let $P\in V(J)$. Then $J_1\cap J_2\subseteq \rad(J)\subseteq P$. By primeness of $P$, $J_1\subseteq P$ or $J_2\subseteq P$, whence $P\in V(J_1)\cup V(J_2)$. Therefore, $V(J)\subseteq V(J_1)\cup V(J_2)$. Since $V(J)$ is irreducible, $V(J)\subseteq V(J_1)$ or $V(J)\subseteq V(J_2)$. If $V(J)\subseteq V(J_1)$, then $J_1\subseteq P$ for all $P\in V(J)$, whence $J_1\subseteq \rad(J)$. Similarly, if $V(J)\subseteq V(J_2)$, then $J_2\subseteq \rad(J)$. Hence, $\rad(J)$ is a prime $\Gamma$-order ideal. Now suppose $\rad(J)$ is prime. Then $\rad(J)\neq M$ and so $V(J)\neq \emptyset$. To show that $V(J)$ is irreducible, let $I_1,I_2$ be $\Gamma$-order ideals such that $V(J)=V(I_1)\cup V(I_2)$. Then $V(J)=V(I_1\cap I_2)$. Since $\rad(J)$ is a prime $\Gamma$-order ideal containing $J$, $I_1\cap I_2\subseteq \rad(J)$. The primeness of $\rad(J)$ now implies that $I_1\subseteq \rad(J)$ or $I_2\subseteq \rad(J)$. If the first one happens, then for any $P\in V(J)$, $I_1\subseteq P$ and so $P\in V(I_1)$. Hence, $V(J)=V(I_1)$. Similarly, if $I_2\subseteq \rad(J)$, then $V(J)=V(I_2)$. This shows that $V(J)$ is irreducible.  

$(vii)$ Let $X$ be any nonempty irreducible subset of $\Spec_\Gamma(M)$. Then $X=V(J)$ for some $\Gamma$-order ideal $J$ of $M$ such that $\rad(J)$ is prime. Now using $(iii)$ and $(v)$, we have $V(J)=V(\rad(J))=\overline{\{J\}}$, showing that $\rad(J)$ is a generic point for $V(J)$. The uniqueness of the generic points is clear since for any two prime $\Gamma$-order ideals $P, Q$, $V (P ) = V (Q)$ implies $P = Q$.

$(viii)$ This follows from $(iv)$ and $(vii)$. 

Now we assume that $M$ has the refinement property.

$(1)$ The `if' direction is evident as no prime $\Gamma$-order ideal can contain an order unit. Now assume that $D(x)=\Spec_\Gamma(M)$. This implies $\langle x\rangle \nsubseteq P$ for every prime $\Gamma$-order ideal $P$. Then $\langle x\rangle=M$, otherwise if it is a proper $\Gamma$-order ideal, then it would be contained in a prime $\Gamma$-order ideal of $M$ by Lemma \ref{lem proper ideal contained in a prime ideal}. It follows that for any $a\in M$, there is some $\alpha\in \mathbb{N}[\Gamma]$ such that $a\le~^{\alpha}x$. Hence, $x$ is an order unit of $M$.

$(2)$ If $P$ is a maximal $\Gamma$-order ideal, then it is prime by Lemma \ref{lem maximal order ideal is prime} and it is clear that $V(P)=\{P\}$. Therefore, $\{P\}$ is closed. Conversely, if $\{P\}$ is closed, then $\{P\}=V(P)$ by $(iii)$. Let $P\nsubseteq I\subseteq M$ be a $\Gamma$-order ideal. If $I\neq M$, then by Lemma \ref{lem proper ideal contained in a prime ideal}, there is a prime $\Gamma$-order ideal $K$ of $M$ such that $I\subseteq K$. But then $K\in V(P)$ and $K\neq P$ which is a contradiction. Therefore, $I=M$ proving that $P$ is maximal. 

$(3)$ Let $J$ be any $\Gamma$-order ideal of $M$. The inclusion $J\subseteq \rad(J)$ is obvious. Since $\rad(M)=M$, we may assume that $J$ is proper. But then the result follows from Lemma \ref{lem proper ideal contained in a prime ideal}. Since $\rad(\{0\})=\nil(M)$, it follows that $\nil(M)=\{0\}$. 

$(4)$ Follows from $(i)$ and $(3)$.

$(5)$ Let $x\in M$. Since the sets $D(y)$, $y\in M$ form a basis for the Zariski topology on $\Spec_\Gamma(M)$, it suffices to check that any open cover of $D(x)$ consisting of the basic open sets of the form $D(y)$ has a finite subcover. So let $\Delta$ be an index set and $\mathcal{U}=\{D(x_\nu)~|~\nu\in \Delta,x_\nu\in M\}$ an open cover of $D(x)$. Let $J=\langle \{x_\nu~|~\nu\in \Delta\}\rangle$. Let $P\in V(J)$. If $x\notin P$, then $P\in D(x)$ and consequently, $P\in D(x_\nu)$ for some $\nu\in \Delta$. But then $J\nsubseteq P$ which is not the case. Thus $x\in P$ for all $P\in V(J)$ and so $x\in \rad(J)=J$ by $(3)$. It follows that $x\le \displaystyle{\sum_{i=1}^{t}}~^{\gamma_i}x_{\nu_i}$ for some $t\in \mathbb{N}$ and $\gamma_i\in \mathbb{N}[\Gamma]$; $i=1,2,\ldots,t$. Now for any $Q\in D(x)$, we must have $j\in \{1,2,\ldots,t\}$ such that $x_{\nu_j}\notin Q$, i.e., $Q\in D(x_{\nu_j})$. Therefore, $D(x)\subseteq \displaystyle{\bigcup_{i=1}^{t}}~D(x_{\nu_i})$ and we have extracted a finite subcover for $D(x)$ from $\mathcal{G}$. Since $\Spec_\Gamma(M)=D(\epsilon)$, $\Spec_\Gamma(M)$ is also compact.

For the final statement, assume that every $\Gamma$-order ideal of $M$ is finitely generated. Note that $\Spec_\Gamma(M)$ is already spectral in the sense of Hofmann--Keimel. It is compact by $(5)$. Also, by the same part, each member of the base $\mathcal{B}=\{D(x)~|~x\in M\}$ is compact. Let $x,y\in M$. Then by $(ii)$, $D(x)\cap D(y)=D\left(\langle x\rangle \cap \langle y\rangle\right)$. By the assumption, there exist $n\in \mathbb{N}$ and $z_i\in \langle x\rangle \cap \langle y\rangle$, $i=1,2,\ldots,n$ such that $\langle x\rangle \cap \langle y\rangle=\langle \{z_i~|~i=1,2,\ldots,n\}\rangle=\langle \displaystyle{\sum_{i=1}^{n}}~z_i\rangle$. It is then easy to see that \[D(x)\cap D(y)=D\left(\langle x\rangle \cap \langle y\rangle\right)=D\left(\displaystyle{\sum_{i=1}^{n}}~z_i\right).\] Therefore, the base $\mathcal{B}$ is closed under finite intersection. Hence, $\Spec_\Gamma(M)$ is spectral in the sense of Hochster. 
\end{proof}

The following is an immediate corollary to the above proposition.

\begin{cor}\label{cor spec of TM is compact}
Let $\Lambda$ be a row-finite $k$-graph without sources. Then $\Spec_{\mathbb{Z}^k}(T_\Lambda)$ is compact. Moreover, if $|\Lambda^0|<\infty$, then $\Spec_{\mathbb{Z}^k}(T_\Lambda)$ is spectral in the sense of Hochster. In particular, for any finite graph $E$, $\Spec_{\mathbb{Z}}(T_E)$ is spectral in the sense of Hochster. 
\end{cor}
\begin{proof}
The first part follows from part $(5)$ of Proposition \ref{pro basic facts about basic open sets} since $T_\Lambda$ has the refinement property by \cite[Theorem 3.18]{HMPS1}. For the second part, let $J$ be any $\mathbb{Z}^k$-order ideal of $T_\Lambda$. Then $J=J_H$ for some hereditary saturated subset $H$ of $\Lambda^0$. Since $|\Lambda^0|<\infty$, $H$ is finite and so $J$ is finitely generated. The result now follows from Proposition \ref{pro basic facts about basic open sets}.
\end{proof}

Next we obtain a necessary condition for the prime spectrum to be connected. 

\begin{prop}\label{pro Necessary condition for connectedness}
Let $\Gamma$ be a group and $M$ a $\Gamma$-monoid which has the refinement property. If $\Spec_\Gamma(M)$ is connected, then $M$ is indecomposable. 
\end{prop}
\begin{proof}
We prove the contrapositive. So assume that $M$ is decomposable. Then there are proper $\Gamma$-order ideals $I$ and $J$ such that $I\cap J=\{0\}$ and $M=\langle I\cup J\rangle$. We claim that $V(I)$ is a nontrivial clopen subset of $\Spec_\Gamma(M)$. It is obviously closed. Let $P\in V(I)$. Then $I\subseteq P$. If $J\subseteq P$, then it will imply $M=\langle I\cup J\rangle\subseteq P$ which is not possible since $P$ is prime and hence proper. Thus, $J\nsubseteq P$ and so $P\in \Spec_\Gamma(M)\setminus V(J)$. Conversely, if $Q\in \Spec_\Gamma(M)\setminus V(J)$, then $J\nsubseteq Q$. Now since $I\cap J=\{0\}\subseteq Q$ and $Q$ is prime, it must happen that $I\subseteq Q$, whence $Q\in V(I)$. Therefore, $V(I)=\Spec_\Gamma(M)\setminus V(J)$. This shows that $V(I)$ is open in $\Spec_\Gamma(M)$. Since $I$ is a proper $\Gamma$-order ideal, by Lemma \ref{lem proper ideal contained in a prime ideal}, $I$ is contained in some prime $\Gamma$-order ideal of $M$. So $V(I)\neq \emptyset$. Similarly, $V(J)\neq \emptyset$ and so $V(I)\neq \Spec_\Gamma(M)$. Hence our claim is established. Therefore, $\Spec_\Gamma(M)$ is not connected. 
\end{proof}
In the next section, we shall see that the converse of the above proposition holds for the talented monoid of a row-finite $k$-graph without sources and with a finite set of objects.

\section{Regular \texorpdfstring{$\Gamma$}{Gamma}-order ideals and the prime spectrum}\label{sec regular ideals}
In \cite{CGHaz}, Cordeiro, Gonc\c alves, and the first author introduced the notion of regular ideals in the setting of $\Gamma$-monoids. Recall that if $J$ is a $\Gamma$-order ideal of a $\Gamma$-monoid $M$, then \[J^{\perp}:=\{x\in M~|~x~||~a~\text{for all}~0\neq a\in J\}.\] When $M$ is a conical refinement monoid, it can be shown that $J^{\perp}$ is also a $\Gamma$-order ideal of $M$ (see \cite[Proposition 2.6]{CGHaz}). We write $J^{\perp \perp}$ to denote $(J^\perp)^\perp$. Clearly, $J\subseteq J^{\perp \perp}$. A $\Gamma$-order ideal $J$ is called \emph{regular} if $J=J^{\perp \perp}$. The set of all regular $\Gamma$-order ideals of $M$ forms a poset with respect to inclusion. We denote this poset by $\reg(M)$. 

The main result of this section is Theorem \ref{th regular ideals via the prime spectrum}, which helps us to realize regular ideals of a conical refinement $\Gamma$-monoid via its prime spectrum. We start by proving two small lemmas. 

\begin{lem}\label{lem V(I)=V(J) implies I=J}
Let $\Gamma$ be a group and $M$ a $\Gamma$-monoid with the refinement property. Let $J_1$, $J_2$ be two $\Gamma$-order ideals of $M$. Then $J_1=J_2$ if and only if $V(J_1)=V(J_2)$ in $\Spec_\Gamma(M)$.    
\end{lem}
\begin{proof}
The forward direction is obvious. So assume that $V(J_1)=V(J_2)$. Now, using Proposition \ref{pro basic facts about basic open sets}(3), we have $J_1=\rad(J_1)=\displaystyle{\bigcap_{P\in V(J_1)}}P=\displaystyle{\bigcap_{P\in V(J_2)}}P=\rad(J_2)=J_2$.
\end{proof}
\begin{lem}\label{lem perp via intersection}
Let $\Gamma$ be a group and $M$ a conical $\Gamma$-monoid with the refinement property. Let $J$ be a $\Gamma$-order ideal of $M$ and $x\in M$. Then $x\in J^\perp$ if and only if $\langle x \rangle \cap J=\{0\}$.     
\end{lem}
\begin{proof}
The result is obvious if $x=0$. So we may assume that $x\neq 0$. Suppose $\langle x\rangle \cap J\neq \{0\}$. Pick any $0\neq a\in \langle x\rangle \cap J$. Then $a\le~^\alpha x$ for some $\alpha\in \mathbb{N}[\Gamma]$. This implies that $^\alpha x\notin J^\perp$. Since $J^\perp$ is a $\Gamma$-order ideal, it is closed under the $\mathbb{N}[\Gamma]$-action. It follows that $x\notin J^\perp$. Conversely, if $x\notin J^\perp$, then $x$ is comparable to some element $y\in J\setminus \{0\}$. Then either $x\le y$ or $y<x$. If $x\le y$, then $x\in \langle x\rangle \cap J$, and consequently, $\langle x\rangle \cap J\neq \{0\}$. Again if $y< x$, then $y\in \langle x\rangle \cap J$ proving that $\langle x \rangle \cap J\neq \{0\}$. Hence the lemma follows.  
\end{proof}
We are now ready to prove the main result of this section.
\begin{thm}\label{th regular ideals via the prime spectrum}
Let $\Gamma$ be a group and $M$ a conical $\Gamma$-monoid with the refinement property. Let $J$ be any $\Gamma$-order ideal of $M$. Then the following are equivalent.

$(i)$ $J$ is regular.

$(ii)$ $V(J)$ is a regular closed set in $\Spec_\Gamma(M)$.

$(iii)$ $D(J)$ is a regular open set in $\Spec_\Gamma(M)$.

The map 
\begin{align*}
\varphi:\reg(M)&\longrightarrow \mathcal{R}\mathcal{O}(\Spec_\Gamma(M))\\
J&\longmapsto D(J)
\end{align*}
is an order-isomorphism, and consequently $\reg(M)$ is a complete Boolean algebra. 
\end{thm}
\begin{proof}
Since $D(J)=\Spec_\Gamma(M)\setminus V(J)$, the equivalence of $(ii)$ and $(iii)$ is obvious. We only show that $(i)\Leftrightarrow (ii)$. For this, we first prove that $V(I^\perp)=\Spec_\Gamma(M)\setminus \Int(V(I))$ for any $\Gamma$-order ideal $I$ of $M$.

Let $P\in V(I^\perp)$. We show that $P$ is not an interior point of $V(I)$. Let $x\in M$ and $D(x)$ any basic open set containing $P$. Then $x\notin P$. Since $P\in V(I^\perp)$, $I^\perp\subseteq P$. Then $x\notin I^\perp$. By Lemma \ref{lem perp via intersection}, we then have $\langle x\rangle \cap I\neq \{0\}$. Since $M$ is a refinement monoid, $\nil(M)=\{0\}$. It follows that there is a prime $\Gamma$-order ideal $Q$ of $M$ such that $\langle x\rangle \cap I\nsubseteq Q$. Then $x\notin Q$ and $I\nsubseteq Q$. Hence, $Q\in D(x)\cap (\Spec_\Gamma(M)\setminus V(I))$. We have shown that there is no basic open set containing $P$ which is entirely contained in $V(I)$. Therefore, $P\in \Spec_\Gamma(M)\setminus \Int(V(I))$. It follows that $V(I^\perp)\subseteq \Spec_\Gamma(M)\setminus \Int(V(I))$. 

For the other inclusion, let $P\notin V(I^\perp)$. Then $I^\perp \nsubseteq P$. Choose any $x\in I^\perp \setminus P$. We claim that $D(x)\subseteq V(I)$. Let $Q\in D(x)$. Consider the $\Gamma$-order ideal $\langle x\rangle \cap I$. Since $x\in I^\perp$, by Lemma \ref{lem perp via intersection}, we have $\langle x\rangle \cap I=\{0\}$. Since $Q$ is prime and $0\in Q$, $\langle x\rangle \subseteq Q$ or $I\subseteq Q$. But $x\notin Q$ since $Q\in D(x)$. Hence, $I\subseteq Q$, whence $Q\in V(I)$. The claim is thus established. Now since $x\notin P$, $D(x)$ is an open set containing $P$, which is entirely contained in $V(I)$. So $P\in \Int(V(I))$. This shows that $\Spec_\Gamma(M)\setminus \Int(V(I))\subseteq V(I^\perp)$. Hence, 
\begin{equation}\label{eq the perp}
V(I^\perp)=\Spec_\Gamma(M)\setminus \Int(V(I)).
\end{equation}
Writing equality (\ref{eq the perp}) first for $I=J^\perp$ and then for $I=J$, we get 
\begin{align*}
V(J^{\perp\perp})&=\Spec_\Gamma(M)\setminus \Int(V(J^\perp))\\
&=\Spec_\Gamma(M)\setminus \Int\left(\Spec_\Gamma(M)\setminus \Int(V(J))\right)\\
&=\overline{\Spec_\Gamma(M)\setminus \left(\Spec_\Gamma(M)\setminus \Int(V(J))\right)}\\
&=\overline{\Int(V(J))}.
\end{align*}
If $J$ is regular, then $J=J^{\perp \perp}$, and we have $V(J)=\overline{\Int(V(J))}$. Hence, $V(J)$ is a regular closed set. On the other hand, if $V(J)$ is regular closed, then $V(J^{\perp \perp})=\overline{\Int(V(J))}=V(J)$. Then $J^{\perp\perp}=J$ by Lemma \ref{lem V(I)=V(J) implies I=J} and hence, $J$ is regular. 

By the first part, it is clear that the map $\varphi$ is well-defined and surjective. If for two regular $\Gamma$-order ideals, $J_1,J_2$, $D(J_1)=D(J_2)$, then $V(J_1)=V(J_2)$ which implies $J_1=J_2$ by Lemma \ref{lem V(I)=V(J) implies I=J}. This shows that $\varphi$ is injective. It clearly preserves the order. Let $J_1,J_2$ be regular $\Gamma$-order ideals such that $D(J_1)\subseteq D(J_2)$. Then $V(J_2)\subseteq V(J_1)$. Now using Proposition \ref{pro basic facts about basic open sets}(3), we have \[J_1=\rad(J_1)=\displaystyle{\bigcap_{P\in V(J_1)}}P\subseteq \displaystyle{\bigcap_{P\in V(J_2)}}P=\rad(J_2)=J_2.\] Therefore, $\varphi$ also reflects the order. Hence, it is an order-isomorphism. Since $\mathcal{R}\mathcal{O}(\Spec_\Gamma(M))$ is a complete Boolean algebra, it now follows that $\reg(M)$ is also a complete Boolean algebra. 
\end{proof}
In view of the order-isomorphism (in fact lattice isomorphism) $\varphi$ of Theorem \ref{th regular ideals via the prime spectrum}, we can explicitly describe the operations on the Boolean algebra $\reg(M)$. 

Let $J,J_1,J_2\in \reg(M)$. We first simplify the join, meet and complement in $\mathcal{R}\mathcal{O}(\Spec_\Gamma(M))$ using (\ref{eq the perp}). Note that 
\begin{align*}
D(J_1)\vee D(J_2)=\Int(\overline{D(J_1)\cup D(J_2)})&=\Int(\overline{D(J_1+J_2)})\\
&=\Int\left(\overline{\Spec_\Gamma(M)\setminus V(J_1+J_2)}\right)\\
&=\Int\left(\Spec_\Gamma(M)\setminus \Int(V(J_1+J_2))\right)\\
&=\Spec_\Gamma(M)\setminus \overline{\Int(V(J_1+J_2))}\\
&=\Spec_\Gamma(M)\setminus V\left((J_1+J_2)^{\perp\perp}\right)=D\left((J_1+J_2)^{\perp\perp}\right).
\end{align*}
Again, $D(J_1)\wedge D(J_2)=D(J_1)\cap D(J_2)=D(J_1\cap J_2)$ and \[D(J)'=\Int(\Spec_\Gamma(M)\setminus D(J))=\Int(V(J))=\Spec_\Gamma(M)\setminus V(J^\perp)=D(J^\perp).\] Therefore, the join, meet, and the complement operations in $\reg(M)$ are given by \[J_1\vee J_2=\varphi^{-1}\left(D(J_1)\vee D(J_2)\right)=(J_1+J_2)^{\perp\perp},\] \[J_1\wedge J_2=\varphi^{-1}\left(D(J_1)\wedge D(J_2)\right)=J_1\cap J_2,\] and \[J'=\varphi^{-1}\left(D(J)'\right)=J^\perp.\] 

It is interesting and sometimes helpful to consider the complements of prime $\Gamma$-order ideals in a $\Gamma$-monoid (for instance, see Theorem \ref{th maximal tails and prime filters}). To understand these complements precisely, we define the notion of $\Gamma$-filters. The reason for putting this discussion here will be apparent from Corollary \ref{cor regular prime order ideals}.  

\begin{dfn}\label{def Gamma filters and prime Gamma filters}
Let $\Gamma$ be a group and $M$ a $\Gamma$-monoid. A nonempty subset $\mathcal{F}$ of $M$ is called a $\Gamma$-\emph{filter} if the following conditions hold:

(F1) $0\notin \mathcal{F}$ and $\mathcal F$ is closed under the action of $\Gamma$;

(F2) for every $x\in \mathcal{F}$, $y\in M$ with $y\ge x$, $y\in \mathcal{F}$; 

(F3) for any two $x,y\in \mathcal{F}$, there exist $\alpha,\beta\in \mathbb{N}[\Gamma]$ and $z\in \mathcal{F}$ such that $^\alpha x,~^\beta y\ge z$. 

Let $\mathcal{F}$ be a $\Gamma$-filter. Then $\mathcal{F}$ is called \emph{prime} if for $x,y\in M$, $x+y\in \mathcal{F}$ implies either $x\in \mathcal{F}$ or $y\in \mathcal{F}$. 
\end{dfn}

Let $M$ be a conical $\Gamma$-monoid. For $x,y\in M$ and $X\subseteq M$, we write $x \Cap_X y$ if there exist $\alpha,\beta\in \mathbb{N}[\Gamma]$ and $0\neq z\in X$ such that $^\alpha x,~^\beta y\ge z$. In view of this notation, (F3) in Definition \ref{def Gamma filters and prime Gamma filters} is equivalent to saying that $x\Cap_\mathcal{F} y$ for all $x,y\in \mathcal{F}$. There is a canonical way to construct $\Gamma$-filters in $M$. For any $\emptyset\neq X\subseteq M$ with the property that $x\Cap_X y$ for all $x,y\in X$, it is easy to check that the set \[\mathcal{F}(X):=\{x\in M~|~^\gamma x\ge y~\text{for some}~\gamma\in \Gamma~\text{and}~y\in X\}\] is a $\Gamma$-filter containing $X$. We call $\mathcal{F}(X)$ the $\Gamma$-filter generated by $X$ (it is the smallest $\Gamma$-filter that contains $X$).

We denote the set of all $\Gamma$-filters (resp., prime $\Gamma$-filters) of $M$ by $\mathscr{F}il(M)$ (resp., $\mathscr{P}\mathscr{F}il(M)$). One can define a topology on $\mathscr{F}il(M)$ as follows: for any $\Gamma$-order ideal $J$ of $M$, define \[\mathscr{D}(J):=\{\mathcal{F}\in \mathscr{F}il(M)~|~J\cap \mathcal{F}\neq \emptyset\}.\] It is then easy to check that the sets $\mathscr{D}(J)$ cover $\mathscr{F}il(M)$ and $\mathscr{D}(J_1)\cap \mathscr{D}(J_2)=\mathscr{D}(J_1\cap J_2)$ for any two $\Gamma$-order ideals $J_1$ and $J_2$. It follows that the sets $\mathscr{D}(J)$ form a basis for a topology on $\mathscr{F}il(M)$.

We consider $\mathscr{P}\mathscr{F}il(M)$ as a topological space under the subspace topology inherited from the topology of $\mathscr{F}il(M)$ defined above. We now state the following theorem which connects prime $\Gamma$-order ideals and prime $\Gamma$-filters.

\begin{prop}\label{pro dual of prime order ideals}
Let $\Gamma$ be a group and $M$ a $\Gamma$-monoid. Let $J$ be a proper subset of $M$. Then $J$ is a prime $\Gamma$-order ideal of $M$ if and only if $\mathcal{F}_J:= M\setminus J$ is a prime $\Gamma$-filter of $M$. Moreover, the map $\phi:\Spec_\Gamma(M)\longrightarrow \mathscr{P}\mathscr{F}il(M)$ defined by $\phi(J):=\mathcal{F}_J$ is a homeomorphism.   
\end{prop}
\begin{proof}
Suppose $J$ is a prime $\Gamma$-order ideal. Since $J\subsetneq M$, $\mathcal{F}_J\neq \emptyset$. Clearly, $0\notin \mathcal{F}_J$. Let $x\in \mathcal{F}_J$, $y\in M$ such that $y\ge x$. If there exists $\gamma\in \Gamma$ such that $^\gamma y\in J$ then $y\in J$ as $J$ is closed under the $\mathbb{Z}^k$-action. This in turn implies $x\in J$ since $J$ is an order ideal. Therefore, $^\gamma y\in \mathcal{F}_J$ for all $\gamma\in \Gamma$. Let $x,y\in \mathcal{F}_J$. Define \[J_x:=\{a\in M~|~a\le~^\alpha x~\text{for some}~\alpha\in \mathbb{N}[\Gamma]\}\] and \[J_y:=\{b\in M~|~b\le~^\beta x~\text{for some}~\beta\in \mathbb{N}[\Gamma]\}.\] Then $J_x,J_y$ are $\Gamma$-order ideals of $M$. Since $x,y\notin J$, $J_x,J_y\nsubseteq J$. Since $J$ is prime, this implies $J_x\cap J_y\nsubseteq J$. Hence, there exists $z\in (J_x\cap J_y)\setminus J$, which amounts to say that there exist $\alpha,\beta\in \mathbb{N}[\Gamma]$ such that $^\alpha x,~^\beta y\ge z$ and $z\in \mathcal{F}_J$. This proves that $\mathcal{F}_J$ is a $\Gamma$-filter of $M$. To show that it is prime, let $x,y\in M$ be such that $x+y\in \mathcal{F}_J$. If both $x,y\in J$, then it would imply $x+y\in J$ since $J$ is a submonoid. Hence, either $x\in \mathcal{F}_J$ or $y\in \mathcal{F}_J$. 

Conversely, assume that $\mathcal{F}_J$ is a prime $\Gamma$-filter of $M$. Let $x,y\in M$ and $\alpha,\beta\in \Gamma$. Suppose $^\alpha x+~^\beta y\in J$. Now if $x\in \mathcal{F}_J$, then by (F2), $^\alpha x+~^\beta y\in \mathcal{F}_J$, a contradiction. So $x\in J$. Similarly, $y\in J$. Now suppose both $x,y\in J$. If $^\alpha x+~^\beta y\in \mathcal{F}_J$, then $^\alpha x\in \mathcal{F}_J$ or $^\beta y\in \mathcal{F}_J$ since $\mathcal{F}_J$ is a prime $\Gamma$-filter. But this would imply $x\in \mathcal{F}_J$ or $y\in \mathcal{F}_J$ by (F2). Hence, $^\alpha x+~^\beta y\in J$. We have shown that $J$ is a $\Gamma$-order ideal. Now suppose $J_1,J_2$ are $\Gamma$-order ideals with $J_1\cap J_2\subseteq J$. If possible suppose neither $J_1\subseteq J$ nor $J_2\subseteq J$. Then there exist $x\in J_1\setminus J$ and $y\in J_2\setminus J$. By (F3), there exist $\alpha,\beta\in \mathbb{N}[\Gamma]$ and $z\in \mathcal{F}_J$ such that $^\alpha x,~^\beta y\ge z$. Since $J_1,_2$ are $\Gamma$-order ideals, we have $z\in (J_1\cap J_2)\setminus J$, a contradiction. Hence, $J_1\subseteq J$ or $J_2\subseteq J$ showing that $J$ is a prime $\Gamma$-order ideal. 

The final part is straightforward by noting that $\phi(D(J))=\mathscr{D}(J)$ for any $\Gamma$-order ideal $J$ of $M$.    
\end{proof}

Fix a conical refinement $\Gamma$-monoid $M$. When $\mathcal{F}$ is a prime $\Gamma$-filter, it is clear from Proposition \ref{pro dual of prime order ideals} that $J_\mathcal{F}:=M\setminus \mathcal{F}$ is a prime $\Gamma$-order ideal. If $x,y\in J_\mathcal{F}^\perp\setminus \{0\}$, then $x,y\in \mathcal{F}$ and so by (F3), there exist $\alpha,\beta\in \mathbb{N}[\Gamma]$ and $z\in \mathcal{F}$ such that $^\alpha x,~^\beta y\ge z$. Since $J_{\mathcal{F}}^\perp$ is a $\Gamma$-order ideal, it follows that $z\in J_\mathcal{F}^\perp\setminus \{0\}$. Therefore, $x\Cap_{J_\mathcal{F}^\perp\setminus \{0\}} y$ for all $x,y\in J_\mathcal{F}^\perp\setminus \{0\}$. Hence if $J_\mathcal{F}^\perp\setminus \{0\}\neq \emptyset$, then $\mathcal{F}(J_\mathcal{F}^\perp\setminus \{0\})$ is a $\Gamma$-filter containing $J_\mathcal{F}^\perp\setminus \{0\}$. Since $J_\mathcal{F}^\perp\setminus \{0\}\subseteq \mathcal{F}$, it is clear that $\mathcal{F}(J_\mathcal{F}^\perp\setminus \{0\})\subseteq \mathcal{F}$. 

We define \[\mathscr{P}\mathscr{F}il_r(M):=\{\mathcal{F}\in \mathscr{P}\mathscr{F}il(M)~|~\mathcal{F}=\mathcal{F}(J_\mathcal{F}^\perp\setminus \{0\})\},\] and \[\reg\Spec_\Gamma(M):=\{P\in \Spec_\Gamma(M)~|~P~\text{is regular}\}.\] We claim that the homeomorphism $\phi$ of Proposition \ref{pro dual of prime order ideals} maps the subspace $\reg\Spec_\Gamma(M)$ onto the subspace $\mathscr{P}\mathscr{F}il_r(M)$. Suppose $P$ is a regular prime $\Gamma$-order ideal. Then $J_{\mathcal{F}_P}^\perp\setminus \{0\}=P^\perp\setminus \{0\}$. It is nonempty since otherwise we would have $P=P^{\perp\perp}=M$, a contradiction to the fact that $P$ is prime. We already have $\mathcal{F}(J_{\mathcal{F}_P}^\perp\setminus \{0\})=\mathcal{F}(P^\perp\setminus \{0\})\subseteq \mathcal{F}_P$. If $x\notin \mathcal{F}(P^\perp\setminus 0)$, then $x\in P^{\perp\perp}=P$. Therefore, we have $\mathcal{F}(P^\perp\setminus \{0\})=\mathcal{F}_P$, which shows that $\phi(\reg\Spec_\Gamma(M))\subseteq \mathscr{P}\mathscr{F}il_r(M)$. Now let $\mathcal{F}$ be a prime $\Gamma$-filter such that $\mathcal{F}=\mathcal{F}(J_\mathcal{F}^\perp\setminus \{0\})$. Then $J_\mathcal{F}$ is a prime $\Gamma$-order ideal and $\phi(J_\mathcal{F})=\mathcal{F}$. We just need to show that $J_\mathcal{F}$ is regular. Let $x\in J_\mathcal{F}^{\perp\perp}$. If $x\in \mathcal{F}$, then $x\in \mathcal{F}(J_\mathcal{F}^\perp\setminus \{0\})$. This implies that $^\gamma x\ge a$ for some $\gamma\in \Gamma$ and $a\in J_\mathcal{F}^\perp\setminus \{0\}$, whence $x\notin J_\mathcal{F}^{\perp\perp}$. This contradicts the assumption. Hence, $x\in M\setminus \mathcal{F}=J_\mathcal{F}$. So $J_\mathcal{F}^{\perp\perp}=J_\mathcal{F}$. 

What we have proved in the above paragraph is summarized in the following corollary.
\begin{cor}\label{cor regular prime order ideals}
Let $\Gamma$ be a group and $M$ a conical refinement $\Gamma$-monoid. Then $\reg\Spec_\Gamma(M)$ is homeomorphic to $\mathscr{P}\mathscr{F}il_r(M)$.     
\end{cor}

\section{The prime spectrum of the talented monoid of a \texorpdfstring{$k$}{k}-graph}\label{sec maximal tail and TM}
With the knowledge of the last two sections under the belt, in this section, we proceed to study the prime spectrum of the talented monoid associated with a row-finite higher-rank graph without sources. As the main result of this section (Theorem \ref{th the three term connection for KPA}), we show that the prime spectrum of the talented monoid is homeomorphic to the graded prime spectrum of the Kumjian--Pask algebra, thus strengthening the ideal-theoretic connection between this combinatorial monoid and higher-rank graph algebras.  

Our first aim is to characterize prime Kumjian--Pask algebras via the talented monoid. The following lemma, characterizing Condition (MT3) (see Definition \ref{def maximal tail}) in terms of the talented monoid, will serve our purpose. 

\begin{lem}\label{lem TM characterization of MT3}
Let $\Lambda$ be a row-finite $k$-graph without sources. Then $\Lambda^0$ satisfies Condition \textup{(MT3)} if and only if for any two $x,y\in T_\Lambda\setminus \{0\}$, there exist  $z \in T_\Lambda\setminus \{0\}$ and $\N,\M\in \mathbb{N}^k$ such that $^\N x,~^\M y\ge z$.    
\end{lem}
\begin{proof}
Suppose $\Lambda^0$ satisfies (MT3). Let $x,y\in T_\Lambda\setminus \{0\}$. Then there exist vertices $u,v\in \Lambda^0$ and $\bP,\Q\in \mathbb{N}^k$ such that $x\ge u(\bP)$ and $y\ge v(\Q)$. Now by (MT3), there exists $w\in \Lambda^0$ such that $u\Lambda w,v\Lambda w\neq \emptyset$. Pick $\lambda\in u\Lambda w$ and $\mu\in v\Lambda w$. Then in $T_\Lambda$, we have \[u(0)=\displaystyle{\sum_{\alpha\in u\Lambda^{d(\lambda)}}} s(\alpha)(d(\lambda))\ge w(d(\lambda))\] and \[v(0)=\displaystyle{\sum_{\beta\in v\Lambda^{d(\mu)}}} s(\beta)(d(\mu))\ge w(d(\mu)).\] This implies $x\ge u(\bP)\ge w(d(\lambda)+\bP)$ and $y\ge v(\Q)\ge w(d(\mu)+\Q)$, whence $^{(d(\mu)+\Q)}x\ge w(d(\lambda)+d(\mu)+\bP+\Q)$ and $^{(d(\lambda)+\bP)}y\ge w(d(\lambda)+d(\mu)+\bP+\Q)$. Letting $\N=d(\mu)+\Q$, $\M=d(\lambda)+\bP$ and $z=w(d(\lambda)+d(\mu)+\bP+\Q)$, we are done. 

Conversely,  assume that the condition holds in the talented monoid. Let $u,v\in \Lambda^0$. Then there exist $a\in T_\Lambda\setminus \{0\}$ and $\N,\M\in \mathbb{N}^k$ such that $u(\N),v(\M)\ge a$. Then $u(\N)=a+x$ and $v(\M)=a+y$ for some $x,y\in T_\Lambda$. Bypassing these equalities in $M_{\overline{\Lambda}}$ and using the Confluence lemma (\cite[Lemma 3.5]{HMPS1}), we have $\eta,\delta\in \mathbb{F}_{\overline{\Lambda}}\setminus \{0\}$ such that $(u,\N)\longrightarrow \eta$, $a+x\longrightarrow \eta$, $(v,\M)\longrightarrow \delta$, and $a+y\longrightarrow \delta$. By \cite[Lemma 3.1]{HMPS1}, we can write $\eta=\eta_1+\eta_2$, $\delta=\delta_1+\delta_2$ for some $\eta_1,\delta_1\in \mathbb{F}_{\overline{\Lambda}}\setminus \{0\}$ and $\eta_2,\delta_2\in \mathbb{F}_{\overline{\Lambda}}$ such that $a\longrightarrow \eta_1,\delta_1$, $x\longrightarrow \eta_2$ and $y\longrightarrow \delta_2$. Then $\eta_1=\delta_1$ in $M_{\overline{\Lambda}}$ and hence, by the Confluence lemma, there exists $\gamma\in \mathbb{F}_{\overline{\Lambda}}\setminus \{0\}$ such that $\eta_1,\delta_1\longrightarrow \gamma$. Now, we have $(u,\N)\longrightarrow \gamma+\eta_2$ and $(v,\M)\longrightarrow \gamma+\delta_2$. Choose any vertex $(w,\Q)$ in the support of $\gamma$. Then there are paths $\lambda\in u\Lambda^{\Q-\N}w$ and $\mu\in v\Lambda^{\Q-\M}w$, and hence, $u\Lambda w,v\Lambda w\neq \emptyset$. So $\Lambda^0$ has Condition (MT3).
\end{proof}

\begin{cor}\label{cor TM characterization of prime KPA}
Let $\Lambda$ be a row-finite $k$-graph with no sources. Then the following are equivalent.

$(i)$ $\KP(\Lambda)$ is a prime ring.

$(ii)$ For $x,y\in T_\Lambda\setminus \{0\}$, there exist $z\in T_\Lambda\setminus \{0\}$ and $\N,\M\in \mathbb{N}^k$ such that $^\N x,^\M y\ge z$.

$(iii)$ $\{0\}$ is a prime $\mathbb{Z}^k$-order ideal of $T_\Lambda$.

Consequently, if $\Lambda$ and $\Omega$ are row-finite $k$-graphs without sources such that $T_\Lambda\cong T_\Omega$ as $\mathbb{Z}^k$-monoids, then $\KP(\Lambda)$ is prime if and only if $\KP(\Omega)$ is prime.
\end{cor}

\begin{proof}
The equivalence of $(i)$ and $(ii)$ follows from \cite[Theorem 3.1]{Larki} and Lemma \ref{lem TM characterization of MT3}, while $(ii)\Longleftrightarrow (iii)$ in view of \cite[Proposition 4.6]{HMPS1}. Since a $\mathbb{Z}^k$-monoid isomorphism preserves prime $\mathbb{Z}^k$-order ideals, the final assertion follows immediately.
\end{proof}

Motivated by the preceding corollary, we now set up a correspondence between the set of all maximal tails of a $k$-graph $\Lambda$ and the set of prime $\mathbb{Z}^k$-order ideals of the talented monoid $T_\Lambda$. By $\mathscr{M}(\Lambda)$, we denote the poset of all maximal tails of $\Lambda$ ordered by set inclusion. Note that the space $\Spec_{\mathbb{Z}^k}(T_\Lambda)$ also forms a poset under inclusion.
\begin{thm}\label{th maximal tails via TM}
Let $\Lambda$ be a row-finite $k$-graph with no sources. Then the map
\begin{align*}
\phi:\mathscr{M}(\Lambda) & \longrightarrow  \Spec_{\mathbb{Z}^k}(T_\Lambda)\\
M &\longmapsto J_{\Lambda^0\setminus M}
\end{align*}
is an anti-isomorphism (order-reversing bijection) with inverse given by 
\begin{align*}
\psi:\Spec_{\mathbb{Z}^k}(T_\Lambda) &\longrightarrow \mathscr{M}(\Lambda)\\
J &\longmapsto \Lambda^0\setminus H_J.
\end{align*}    
\end{thm}
\begin{proof}
Let $M$ be a maximal tail and $x,y\in T_\Lambda\setminus J_{\Lambda^0\setminus M}$. Then $x,y\neq 0$ and we can write $x=\displaystyle{\sum_{i=1}^{p}}u_i(\M)$ and $y=\displaystyle{\sum_{j=1}^{q}} v_j(\N)$ for some $p,q\ge 1$, vertices $u_i,v_j\in \Lambda^0$ and $\M,\N\in \mathbb{N}^k$. Now if $u_i\in \Lambda^0\setminus M$ for all $i=1,2,\ldots,p$, then it would imply $x\in J_{\Lambda^0\setminus M}$. Therefore, there exists $i\in \{1,2,\ldots,p\}$ such that $u_i\in M$. Similarly, there exists $j\in \{1,2,\ldots,q\}$ such that $v_j\in M$. Since $M$ is a maximal tail, there exists $w\in M$ such that $u_i\Lambda w,v_j\Lambda w\neq \emptyset$. Choose $\lambda\in u_i\Lambda w$ and $\mu\in v_j\Lambda w$. Then it follows that $x\ge u_i(\M)\ge w(\M+d(\lambda))$ and $y\ge v_j(\N)\ge w(\N+d(\mu))$ (see the proof of Lemma \ref{lem TM characterization of MT3}). Let $z=w(d(\lambda)+d(\mu)+\M+\N)$. If $z\in J_{\Lambda^0\setminus M}$, then $w\in J_{\Lambda^0\setminus M}\cap \Lambda^0=\Lambda^0\setminus M$ by \cite[Lemma 4.4]{HMPS1} since $\Lambda^0\setminus M$ is hereditary and saturated. This is a contradiction since $w\in M$. Therefore, $z\in T_\Lambda\setminus J_{\Lambda^0\setminus M}$ and we have $^{\N+d(\mu)}x\ge z$, $^{\M+d(\lambda)}y\ge z$. By \cite[Proposition 4.6]{HMPS1}, it follows that $J_{\Lambda^0\setminus M}$ is a prime $\mathbb{Z}^k$-order ideal of $T_\Lambda$. Hence, the map $\phi$ is well-defined. On the other hand, if $J$ is a prime $\mathbb{Z}^k$-order ideal, then $H_J$ is a hereditary saturated subset. Consequently, $\Lambda^0\setminus H_J$ satisfies (MT1) and (MT2) of Definition \ref{def maximal tail}. Let $u,v\in \Lambda^0\setminus H_J$. Then $u(0),v(0)\in T_\Lambda\setminus J$. Since $J$ is prime, applying \cite[Proposition 4.6]{HMPS1} again, we have $\N,\M\in \mathbb{Z}^k$ and $a\in T_\Lambda\setminus J$ such that $u(\N),v(\M)\ge a$ in $T_\Lambda$. Now, following the same line of argument used in the proof of Lemma \ref{lem TM characterization of MT3}, we can obtain $\gamma\in \mathbb{F}_{\overline{\Lambda}}\setminus \{0\}$ and $x,y\in \mathbb{F}_{\overline{\Lambda}}$ such that $u(\N)\longrightarrow \gamma+x$, $v(\M)\longrightarrow \gamma+y$ and $a\longrightarrow\gamma$. Let $\gamma=\displaystyle{\sum_{t=1}^{l}}(w_i,\Q)$. If $w_t\in H_J$ for all $t=1,2,\ldots,l$, then $a$ would be in $\in J_{H_J}=J$ since $a\longrightarrow \gamma$. So there must exist some $t\in \{1,2,\ldots,l\}$ such that $w_t\in \Lambda^0\setminus H_J$. Now since $u(\N)\longrightarrow \gamma+x$ and $v(\M)\longrightarrow \gamma+y$ and $(w_t,\Q)$ is in the support of $\gamma$, $u\Lambda w_t,v\Lambda w_t\neq \emptyset$. Hence, $\Lambda^0\setminus H_J$ satisfies Condition (MT3), which implies $\Lambda^0\setminus H_J$ is a maximal tail and $\psi$ is well-defined. Using the lattice isomorphism of \cite[Proposition 4.5]{HMPS1}, it is easy to note that $\phi\circ \psi=\text{id}_{\Spec_{\mathbb{Z}^k}(T_\Lambda)}$ and $\psi\circ \phi=\text{id}_{\mathscr{M}(\Lambda)}$. It remains to show that $\phi$ and $\psi$ are order-reversing. If $M_1,M_2$ are maximal tails with $M_1\subseteq M_2$, then $\Lambda^0\setminus M_2\subseteq \Lambda^0\setminus M_1$ and hence, $\phi(M_2)=J_{\Lambda^0\setminus M_2}\subseteq J_{\Lambda^0\setminus M_1}=\phi(M_1)$. The case of $\psi$ is equally straightforward. Hence the theorem follows. 
\end{proof}

We wish to convert the anti-isomorphism of Theorem \ref{th maximal tails via TM} into an order-isomorphism by considering the dual poset of $\Spec_{\mathbb{Z}^k}(T_\Lambda)$. To accomplish this, we make use of the dual poset of $\Spec_{\mathbb{Z}^k}(T_\Lambda)$, namely the poset of all prime $\mathbb{Z}^k$-filters of $T_\Lambda$ (see Section \ref{sec regular ideals}). 
 
\begin{thm}\label{th maximal tails and prime filters}
Let $\Lambda$ be a row-finite $k$-graph with no sources.

$(i)$ For any maximal tail $M$ of $\Lambda$, the set $\mathcal{F}(M):=\{x\in T_\Lambda~|~x\ge v(\N)~\text{for some}~v\in M~\text{and}~\N\in \mathbb{Z}^k\}$ is a prime $\mathbb{Z}^k$-filter of $T_\Lambda$. Moreover, $\mathcal{F}(M)=T_\Lambda\setminus J_{\Lambda^0\setminus M}$. 

$(ii)$ For any prime $\mathbb{Z}^k$-filter $\mathcal{F}$ of $T_\Lambda$, the set $M(\mathcal{F}):=\{v\in \Lambda^0~|~v(0)\in \mathcal{F}\}$ is a maximal tail in $\Lambda$. Moreover, $M(\mathcal{F})=\Lambda^0\setminus H_{T_\Lambda\setminus \mathcal{F}}$. 

$(iii)$ The map $\varrho: \mathscr{M}(\Lambda)\longrightarrow \mathscr{P}\mathscr{F}il(T_\Lambda)$ via $M\longmapsto \mathcal{F}(M)$ is an order isomorphism with inverse given by $\vartheta: \mathscr{P}\mathscr{F}il(T_\Lambda)\longrightarrow \mathscr{M}(\Lambda)$ via $\mathcal{F}\longmapsto M(\mathcal{F})$. Consequently, $M$ is a maximal tail if and only if $\mathcal{F}(M)$ is a prime $\mathbb{Z}^k$-filter of $T_\Lambda$. 
\end{thm}
\begin{proof}
For the first statements of $(i)$ and $(ii)$, in view of Theorem \ref{th maximal tails via TM} and Proposition \ref{pro dual of prime order ideals}, it suffices to prove the second statements. Let $M$ be a maximal tail and $x\in \mathcal{F}(M)$. Then there exist $v\in M$ and $\N\in \mathbb{Z}^k$ such that $x\ge v(\N)$. Now if $x\in J_{\Lambda^0\setminus M}$, then $v=~^{-\N}v(\N)\in J_{\Lambda^0\setminus M}$, which implies $v\in H_{J_{\Lambda^0\setminus M}}=\Lambda^0\setminus M$, a contradiction. Hence, $x\in T_\Lambda\setminus J_{\Lambda^0\setminus M}$. Now suppose $x\in T_\Lambda\setminus J_{\Lambda^0\setminus M}$. Write $x=\displaystyle{\sum_{i=1}^{t}}v_i(\M)$ for vertices $v_i$ and $\M\in \mathbb{N}^k$. If $v_i\notin M$ for all $i=1,2,\ldots,t$, then $x$ would be in $J_{\Lambda^0\setminus M}$ which is not the case. So there must exist $j\in \{1,2,\ldots,t\}$ such that $v_j\in M$. Then $x\ge v_j(\M)$ and so $x\in \mathcal{F}(M)$. This implies $\mathcal{F}(M)=T_\Lambda\setminus J_{\Lambda^0\setminus M}$ and thus proves $(i)$. Now 
\begin{align*}
M(\mathcal{F})=\{v\in \Lambda^0~|~v(0)\in \mathcal{F}\}=\{v\in \Lambda^0~|~v(0)\notin T_\Lambda\setminus \mathcal{F}\}=\Lambda^0\setminus H_{T_\Lambda\setminus \mathcal{F}}.
\end{align*} 
Hence $(ii)$ is proved. Finally $(iii)$ follows from Theorem \ref{th maximal tails via TM} and parts $(i)$ and $(ii)$. 
\end{proof}

\begin{rmk}\label{rem maximal tail axioms resemble prime filters}
One can also prove the first statements of parts $(i)$ and $(ii)$ in Theorem \ref{th maximal tails and prime filters} directly without referring to the complements. In fact, the direct proofs reveal how the three conditions of maximal tails (see Definition \ref{def maximal tail}) mimic the defining conditions of a prime $\mathbb{Z}^k$-filter (see Definition \ref{def Gamma filters and prime Gamma filters}). It is easy to observe that Conditions (MT1) and (MT3) of a maximal tail $M$ result in the upward closedness (F2) and the downward directedness (F3), respectively and vice-versa. Also it is not hard to see that Condition (MT2) is equivalent to the primeness of a $\mathbb{Z}^k$-filter.
\end{rmk}

In \cite{BPRS}, the authors defined a certain $T_0$ topology on the space of all maximal tails of a row-finite directed graph and then in \cite{KangPask}, the same topology was introduced on the space of all maximal tails of a row-finite strongly aperiodic $k$-graph. We now recall this topology.

Let $\Lambda$ be a strongly aperiodic row-finite $k$-graph without sources. Then there is a topology $\tau$ on $\mathscr{M}(\Lambda)$, where the closed sets are of the form \[F_S:=\{M\in \mathscr{M}(\Lambda)~|~M\subseteq \displaystyle{\bigcup_{N\in S}}~N\}\] for some $S\subseteq \mathscr{M}(\Lambda)$. Then $\mathscr{M}(\Lambda)$ equipped with this topology is homeomorphic to the primitive ideal space of $C^*(\Lambda)$ equipped with the Jacobson topology (hull-kernel topology) (see \cite[Theorem 6.3]{BPRS} \& \cite[Theorem 3.15]{KangPask}). In view of \cite[Theorem 5.1]{Pino} and \cite[Proposition 4.2]{Larki}, it follows that $M\longmapsto I(\Lambda^0\setminus M)$ is a bijective correspondence between $\mathscr{M}(\Lambda)$ and $\Prim(\KP(\Lambda))$. One can adopt the same proof as that of \cite[Theorem 6.3]{BPRS} to show that this correspondence is indeed a homeomorphism when $\Prim(\KP(\Lambda))$ is equipped with the hull-kernel topology. 

We now show that the bijection of Theorem \ref{th maximal tails and prime filters}$(iii)$ is also a homeomorphism when we view $\mathscr{P}\mathscr{F}il(T_\Lambda)$ as a topological space with respect to the subspace topology inherited from $\mathscr{F}il(T_\Lambda)$. 

\begin{thm}\label{th TM detects primitive ideal space}
Let $\Lambda$ be a row-finite $k$-graph with no sources. Then the map $\varrho:\mathscr{M}(\Lambda)\longrightarrow \mathscr{P}\mathscr{F}il(T_\Lambda)$; $M\longmapsto \mathcal{F}(M)$ is a homeomorphism. Consequently, if $\Lambda$ and $\Omega$ are strongly aperiodic row-finite $k$-graphs with no sources such that $T_\Lambda\cong T_\Omega$ as $\mathbb{Z}^k$-monoids, then $\Prim(C^*(\Lambda))$ (resp., $\Prim(\KP(\Lambda))$) is homeomorphic to $\Prim(C^*(\Omega))$ (resp., $\Prim(\KP(\Omega))$), and hence, $C^*(\Lambda)$ (resp., $\KP(\Lambda)$) is primitive if and only if $C^*(\Omega)$ (resp., $\KP(\Omega)$) is primitive. 
\end{thm}
\begin{proof}
We already know that $\varrho$ is a bijection. Let $J$ be any $\mathbb{Z}^k$-order ideal of $T_\Lambda$. Let $S:=\{N\in \mathscr{M}(\Lambda)~|~N\subseteq \Lambda^0\setminus H_J\}$. Then
\allowdisplaybreaks
{\begin{align*}
\varrho^{-1}(\mathscr{D}(J))&=\{M\in \mathscr{M}(\Lambda)~|~J\cap \mathcal{F}(M)\neq \emptyset\}\\
&=\{M\in \mathscr{M}(\Lambda)~|~J\nsubseteq T_\Lambda\setminus \mathcal{F}(M)\}\\
&=\{M\in \mathscr{M}(\Lambda)~|~M\nsubseteq \Lambda^0\setminus H_J\}\\
&=\{M\in \mathscr{M}(\Lambda)~|~M\nsubseteq \displaystyle{\bigcup_{N\in S}}~N\}\\
&=\mathscr{M}(\Lambda)\setminus F_S.
\end{align*}
}
Therefore, $\varrho^{-1}(\mathscr{D}(J))$ is open in $\mathscr{M}(\Lambda)$ and so $\varrho$ is continuous. Now we show that it is open. Let $S\subseteq \mathscr{M}(\Lambda)$ and $H:=\displaystyle{\bigcap_{N\in S}}\Lambda^0\setminus S$. Clearly, $H$ is a hereditary saturated subset of $\Lambda^0$. Now, 
\begin{align*}
\varrho(\mathscr{M}(\Lambda)\setminus F_S)=\{\mathcal{F}(M)~|~M\nsubseteq \displaystyle{\bigcup_{N\in S}}~N\}=\{\mathcal{F}\in \mathscr{P}\mathscr{F}il(T_\Lambda)~|~\mathcal{F}\cap J_H\neq \emptyset\}.
\end{align*}
The last equality follows since $M\nsubseteq \displaystyle{\bigcup_{N\in S}}~N\Leftrightarrow H\nsubseteq \Lambda^0\setminus M\Leftrightarrow J_H\nsubseteq J_{\Lambda^0\setminus M}\Leftrightarrow J_H\cap \mathcal{F}(M)\neq \emptyset$ for any $M\in \mathscr{M}(\Lambda)$ and $\mathcal{F}=\mathcal{F}(M(\mathcal{F}))$ for any $\mathcal{F}\in \mathscr{P}\mathscr{F}il(T_\Lambda)$. This shows that $\varrho$ is an open map. Hence, $\varrho$ is a homeomorphism. 
\end{proof}
\begin{thm}\label{th the three term connection for KPA}
Let $\Lambda$ be a row-finite $k$-graph without sources. Then all the spaces $\mathscr{M}(\Lambda)$, $\mathscr{P}\mathscr{F}il(T_\Lambda)$, $\Spec_{\mathbb{Z}^k}(T_\Lambda)$, and $\Spec^{\gr}(\KP(\Lambda))$ are homeomorphic. Moreover, if $|\Lambda^0|<\infty$, then each of these spaces is spectral in the sense of Hochster.
\end{thm}
\begin{proof}
Using \cite[Proposition 4.1]{Larki}, it can be shown that $\mathscr{M}(\Lambda)$ and $\Spec^{\gr}(\KP(\Lambda))$ are homeomorphic via the homeomorphism sending $M$ in $\mathscr{M}(\Lambda)$ to $I(\Lambda^0\setminus M)$ in $\Spec^{\gr}(\KP(\Lambda))$. The result follows from Proposition \ref{pro dual of prime order ideals} and Theorem \ref{th maximal tails and prime filters}. The final statement follows from Corollary \ref{cor spec of TM is compact}.
\end{proof}

We now specialize Theorem \ref{th the three term connection for KPA} for Leavitt path algebras. Suppose $E$ is a row-finite directed graph (possibly with sinks). Following \cite[Definition 4.1.6]{lpabook}, let $\mathscr{M}_\gamma(E)$ denote the subspace of $\mathscr{M}(E)$ consisting of all maximal tails $M$ with the property that any cycle with vertices in $M$ has an exit in $M$. Also, let $\Prim^{\gr}(L(E))$ be the subspace of $\Spec^{\gr}(L(E))$ consisting of the graded prime ideals of $L(E)$ which are primitive. 

Note that the proof of Theorem \ref{th the three term connection for KPA} works equally well for $E$. More precisely, we have homeomorphisms $\phi:\mathscr{M}(E)\longrightarrow \Spec_\mathbb{Z}(T_E)$; $M\longmapsto J_{E^0\setminus M}$ and $\rho:\mathscr{M}(E)\longrightarrow \Spec^{\gr}(L(E))$; $M\longmapsto I(E^0\setminus M)$. Moreover, from \cite[Lemma 2.2 \& Corollary 4.3(i)]{hazli} and \cite[Theorem 4.1.11]{lpabook}, it follows that the restrictions of $\phi$ and $\rho$ to the subspace $\mathscr{M}_\gamma(E)$ map this space onto the subspaces $\Prim_{\mathbb{Z}}(T_E)$ of $\Spec_{\mathbb{Z}}(T_E)$ and $\Prim^{\operatorname{gr}}(L(E))$ of $\Spec^{\operatorname{gr}}(L(E))$, respectively. As a consequence, we have the following commutative diagram of topological spaces and homeomorphisms at the level of Leavitt path algebras: 

\[
\begin{tikzcd}[row sep=huge, column sep=huge]
\Prim_{\mathbb{Z}}(T_E) \arrow[d, hook] & 
\mathscr{M}_{\gamma}(E) \arrow[l, "\cong"'] \arrow[r, "\cong"] \arrow[d, hook] & 
\Prim^{\operatorname{gr}}(L(E)) \arrow[d, hook] \\
\Spec_{\mathbb{Z}}(T_E) & 
\mathscr{M}(E) \arrow[l, "\cong"', "\phi"] \arrow[r, "\cong", "\rho"'] & 
\Spec^{\operatorname{gr}}(L(E)) \\[-3.5em]
J_{E^0 \setminus M} & 
M \arrow[l, maps to] \arrow[r, maps to] & 
I(E^0 \setminus M)
\end{tikzcd}
\]

In Proposition \ref{pro Necessary condition for connectedness}, we saw that if the prime spectrum of a refinement $\Gamma$-monoid is connected, then the monoid is indecomposable. We end this section by showing that the converse also holds for the talented monoids of certain $k$-graphs. Recall that a $k$-graph is called \emph{connected} if the equivalence relation $\approx$ on $\Lambda^0$ generated by $\{(u,v)\in \Lambda^0\times \Lambda^0~|~u\Lambda v\neq \emptyset\}$ coincides with the universal relation on $\Lambda^0$.

\begin{thm}\label{th connectedness of spectrum of TM}
Let $\Lambda$ be a row-finite $k$-graph without sources. Then the following statements are equivalent. 

$(i)$ $\Spec_{\mathbb{Z}^k}(T_\Lambda)$ is connected as a topological space.

$(ii)$ $T_\Lambda$ is indecomposable as a $\mathbb{Z}^k$-monoid.

$(iii)$ There is no pair $(H_1,H_2)\in \mathcal{H}_\Lambda\times \mathcal{H}_\Lambda$ such that $H_1,H_2\neq \emptyset$, $H_1\cap H_2=\emptyset$, and $\overline{H_1\cup H_2}=\Lambda^0$. 

Moreover, if any one of the above holds, then $\Lambda$ is connected as a $k$-graph.
\end{thm}
\begin{proof}
$(i)\Rightarrow (ii)$ Since $T_\Lambda$ is a refinement monoid, this implication directly follows from Proposition \ref{pro Necessary condition for connectedness}.

$(ii)\Rightarrow (i)$ Let $J$ be any nonzero $\mathbb{Z}^k$-order ideal of $T_\Lambda$. We first observe that $V(J)\neq \Spec_{\mathbb{Z}^k}(T_\Lambda)$. Choose any $0\neq a\in J$. Then $a\ge v(\N)$ for some $v\in \Lambda^0$ and $\N\in \mathbb{Z}^k$. Clearly, $v(0)=~^{-\N}v(\N)\in J$. Since $\Lambda$ has no sources, we can choose an infinite path $x\in v\Lambda^\infty$. Let \[M:=\{w\in \Lambda^0~|~w\Lambda x(\M)\neq \emptyset~\text{for some}~\M\in \mathbb{N}^k\}.\] We claim that $M$ is a maximal tail of $\Lambda$. Clearly, $M\neq \emptyset$ since $v=x(0)\in M$. Let $w\in M$ and $u\in \Lambda^0$ be such that $u\Lambda w\neq \emptyset$. Since $w\Lambda x(\M)\neq \emptyset$ for some $\M\in \mathbb{N}^k$, it follows that $u\Lambda x(\M)\neq \emptyset$. Hence, $u\in M$. So $M$ satisfies (MT1). Let $w\in M$ and $i\in \{1,2,\ldots,k\}$. Then $w\Lambda x(\M)\neq \emptyset$ for some $\M\in \mathbb{N}^k$. Choose $\lambda\in w\Lambda x(\M)$. Let $\mu=\lambda x(\M,\M+\mathbf{e}_i)$. Then $\mu(0,\mathbf{e}_i)\in w\Lambda^{\mathbf{e}_i}$ and $s(\mu(0,\mathbf{e}_i))\Lambda x(\M+\mathbf{e}_i)\neq \emptyset$ since $\mu(\mathbf{e}_i,d(\mu))\in s(\mu(0,\mathbf{e}_i))\Lambda x(\M+\mathbf{e}_i)$. So $M$ satisfies (MT2). Let $w_1,w_2\in M$. Then there are $\bP,\Q\in \mathbb{N}^k$ such that $w_1\Lambda x(\bP),w_2\Lambda x(\Q)\neq \emptyset$. Choose $\alpha\in w_1\Lambda x(\bP)$ and $\beta\in w_2\Lambda x(\Q)$. Let $\M:=\bP\vee \Q$. Then $\alpha x(\bP,\M)\in w_1\Lambda x(\M)$ and $\beta x(\Q,\M)\in w_2\Lambda x(\M)$. Since $x(\M)\in M$, $M$ satisfies condition (MT3). Hence, $M$ is a maximal tail. By Theorem \ref{th maximal tails and prime filters}, $J_{\Lambda^0\setminus M}$ is a prime $\mathbb{Z}^k$-order ideal of $T_\Lambda$ and since $v\in M$, $v(0)\notin J_{\Lambda^0\setminus M}$. Therefore, $J\nsubseteq J_{\Lambda^0\setminus M}$, showing that $V(J)\neq \Spec_{\mathbb{Z}^k}(T_\Lambda)$. Now consider any two $\mathbb{Z}^k$-order ideals $J_1,J_2$ of $T_\Lambda$ such that $V(J_1)\cap V(J_2)=\emptyset$. Then $V\left(\langle J_1 \cup J_2\rangle \right)=\emptyset$. This implies that $\langle J_1 \cup J_2\rangle=T_\Lambda$ since any proper $\mathbb{Z}^k$-order ideal of $T_\Lambda$ is contained in a prime $\mathbb{Z}^k$-order ideal by Lemma \ref{lem proper ideal contained in a prime ideal}. Now since $T_\Lambda$ is indecomposable, we must have $J_1\cap J_2\neq \{0\}$. Then $V(J_1\cap J_2)\neq \Spec_{\mathbb{Z}^k}(T_\Lambda)$ by our earlier observation. Hence, $V(J_1)\cup V(J_2)\neq \Spec_{\mathbb{Z}^k}(T_\Lambda)$. Therefore, $\Spec_{\mathbb{Z}^k}(T_\Lambda)$ cannot be written as the union of two disjoint closed sets and so it is connected.

$(ii)\Leftrightarrow (iii)$ This directly follows from the lattice isomorphism between $\mathcal{H}_\Lambda$ and the lattice of all $\mathbb{Z}^k$-order ideals of $T_\Lambda$ (see \cite[Proposition 4.5]{HMPS1}).  

Suppose $(ii)$ holds. If possible, assume that $\Lambda$ is not connected. Then there exists a vertex $v\in \Lambda^0$ such that $[v]_{\approx}$ is not the whole of $\Lambda^0$. Let $I:=\langle \{u(0)~|~u\in [v]_{\approx}\}\rangle$ and $J:=\langle \{w(0)~|~w\in \Lambda^0\setminus [v]_{\approx}\}\rangle$. We show that both $I$ and $J$ are proper $\mathbb{Z}^k$-order ideals of $T_\Lambda$. Choose any $w\in \Lambda^0\setminus [v]_{\approx}$. If $I=T_\Lambda$, then $w(0)\in I$. Hence, $w(0)\le \displaystyle{\sum_{i=1}^{t}}u_i(\N_i)$ for some $t\in \mathbb{N}$, $\N_i\in \mathbb{Z}^k$, and $u_i\in [v]_{\approx}$. Then $\displaystyle{\sum_{i=1}^{t}}u_i(\N_i)=w(0)+a$ for some $a\in T_\Lambda$. Transferring this equality to $M_{\overline{\Lambda}}$ and applying the confluence lemma, we get an element $\gamma\in \mathbb{F}_{\overline{\Lambda}}\setminus \{0\}$ such that $\displaystyle{\sum_{i=1}^{t}}(u_i,\N_i)\longrightarrow \gamma$ and $(w,0)+a \longrightarrow \gamma$. This implies that there exists a vertex $(u,\M)$ in the support of $\gamma$ such that $(w,0)\overline{\Lambda}(u,\M)\neq \emptyset$ and $(u_j,\N_j)\overline{\Lambda} (u,\M)\neq \emptyset$ for some $j\in \{1,2,\ldots,t\}$. It follows that $w\Lambda u,u_j\Lambda u\neq \emptyset$. But then $w\approx u\approx u_j\approx v$, which is a contradiction. Thus, $I$ is a proper $\mathbb{Z}^k$-order ideal of $T_\Lambda$. A similar argument can be applied to show that $v(0)\notin J$ and so $J$ is also proper. Since $H_{\langle I\cup J\rangle}=\Lambda^0$, it follows that $\langle I\cup J\rangle=T_\Lambda$. Finally we observe that $I\cap J=\{0\}$. Indeed if $0\neq a\in I\cap J$, then $a\le \displaystyle{\sum_{i=1}^{p}} u_i(\N_i)$ and $a\le \displaystyle{\sum_{j=1}^{q}}w_j(\M_j)$ for some $p,q\in \mathbb{N}$, $\N_i,\M_j\in \mathbb{Z}^k$, $u_i\in [v]_{\approx}$, and $w_j\in \Lambda^0\setminus [v]_{\approx}$. Then there are $b,c\in M_{\overline{\Lambda}}$ such that \[\displaystyle{\sum_{i=1}^{p}} (u_i,\N_i)=a+b,~\displaystyle{\sum_{j=1}^{q}}(w_j,\M_j)=a+c.\] Applying \cite[Lemma 3.1 \& 3.5]{HMPS1}, we have $\eta_1,\eta_2,\delta_1,\delta_2\in \mathbb{F}_{\overline{\Lambda}}$ with $\eta_1,\delta_1\neq 0$ such that \[\displaystyle{\sum_{i=1}^{p}} (u_i,\N_i)\longrightarrow \eta_1+\eta_2,\displaystyle{\sum_{j=1}^{q}}(w_j,\M_j)\longrightarrow \delta_1+\delta_2,\]  \[a\longrightarrow \eta_1,a\longrightarrow \delta_1,\] \[b\longrightarrow \eta_2,c\longrightarrow \delta_2.\] Then $\eta_1=\delta_1$ in $M_{\overline{\Lambda}}$ and hence, there is some $\gamma\in \mathbb{F}_{\overline{\Lambda}}\setminus \{0\}$ such that $\eta_1,\delta_1\longrightarrow \gamma$. Then \[\displaystyle{\sum_{i=1}^{p}} (u_i,\N_i)\longrightarrow \eta_1+\eta_2\longrightarrow \gamma+\eta_2\] and \[\displaystyle{\sum_{j=1}^{q}}(w_j,\M_j)\longrightarrow \delta_1+\delta_2\longrightarrow \gamma+\delta_2.\] Choose any vertex $(u,\Q)$ in the support of $\gamma$. Then there are $i\in \{1,2,\ldots,p\}$ and $j\in \{1,2,\ldots,q\}$ such that $(u_i,\N_i)\overline{\Lambda}(u,\Q)\neq \emptyset$ and $(w_j,\M_j)\overline{\Lambda}(u,\Q)\neq \emptyset$. But these imply that $u_i\Lambda u,w_j\Lambda u\neq \emptyset$, whence $u_i\approx w_j$ which is not possible. Hence, $I\cap J=\{0\}$ and we have shown that $T_\Lambda$ is decomposable which contradicts $(ii)$. 
\end{proof}
As a joint consequence of Theorem \ref{th the three term connection for KPA} and Theorem \ref{th connectedness of spectrum of TM}, we have the following corollary.
\begin{cor}\label{cor prime spectrum of KPA is connected via TM}
Suppose $\Lambda$ is a row-finite $k$-graph without sources. Then the space $\Spec^{\gr}(\KP(\Lambda))$ is connected if and only if $T_\Lambda$ is indecomposable.  
\end{cor}

\section{Applications}\label{sec applications} 
In this section, we demonstrate some applications of the main results obtained in the previous sections. 

\subsection{Regular ideals of Kumjian--Pask algebras} Recall from \cite{RegLPA} that an ideal $I$ of an algebra $A$ is called \emph{regular} if $I=I^{\perp\perp}$, where \[I^\perp:=\{x\in A~|~xa=ax=0~\text{for all}~a\in I\}\] and $I^{\perp\perp}=(I^\perp)^\perp$. Our first application relates the regular ideals of the talented monoid of a higher-rank graph with the regular graded ideals of the associated Kumjian--Pask algebra. 

For a $k$-graph $\Lambda$, we denote the collection of all regular graded ideals of $\KP(\Lambda)$ by $\reg^{\gr}(\KP(\Lambda))$. The following result can be viewed as an analogue of Theorem \ref{th regular ideals via the prime spectrum} at the level of Kumjian--Pask algebras. The proof follows the same idea with some necessary modifications. But we include the details for the reader's convenience.
\begin{thm}\label{th regular ideal and spectrum of KPA}
Let $\Lambda$ be a row-finite $k$-graph without sources. Let $I$ be any graded ideal of $\KP(\Lambda)$. Then $I$ is a regular ideal if and only if $D^{\gr}(I)$ is a regular open set in $\Spec^{\gr}(\KP(\Lambda))$. Moreover, the poset $\reg^{\gr}(\KP(\Lambda))$ is order-isomorphic to the Boolean algebra $\mathcal{R}\mathcal{O}(\Spec^{\gr}(\KP(\Lambda)))$.     
\end{thm}
\begin{proof}
Let $J$ be any graded ideal of $\KP(\Lambda)$. We claim that \[V^{\gr}(J^\perp)=\Spec^{\gr}(\KP(\Lambda))\setminus \Int(V^{\gr}(J)).\] Choose any $P\in V^{\gr}(J^\perp)$. Choose any open neighborhood of $P$ of the form $D^{\gr}(x)$ for some homogeneous element $x$. Since $x\notin P$ and $J^\perp \subseteq P$, we have $x\notin J^\perp$. So there exists $a\in J$ such that $xa\neq 0$. This implies that $xJ\neq \{0\}$, whence $I(x)\cap J\neq \{0\}$, where $I(x)$ is the graded ideal generated by $x$. Now any graded ideal of $\KP(\Lambda)$ is graded semiprime by \cite[Theorem 6.5(3)]{AHLS}. In particular, $\{0\}$ is graded semiprime meaning that there is no nonzero nilpotent graded ideal (or equivalently, no nonzero graded ideal $J$ such that $J^2=\{0\}$). It follows that $\KP(\Lambda)$ is a graded semiprime ring by \cite[\S 2.11]{NOgraded}. Thus $\displaystyle{\bigcap_{P\in \Spec^{\gr}(\KP(\Lambda))}}P=\{0\}$. Since $I(x)\cap J\neq \{0\}$, there must exist a graded prime ideal $Q$ such that $I(x)\cap J\nsubseteq Q$. Then $Q\in D^{\gr}(x)\cap (\Spec^{\gr}(\KP(\Lambda))\setminus V^{\gr}(J))$, showing that $D^{\gr}(x)$ is not contained in $V^{\gr}(J)$. Since the sets $D^{\gr}(y)$, $y\in \KP(\Lambda)^h$ form a basis for the topology of $\Spec^{\gr}(\KP(\Lambda))$, $P\notin \Int(V^{\gr}(J))$. 

Now let $P\notin V^{\gr}(J^\perp)$. Then there exists $x\in J^\perp \setminus P$. Since $J^\perp$ is graded by \cite[Lemma 4.2]{Schenkel}, we can choose $x$ to be homogeneous. Then $P\in D^{\gr}(x)$. Now since $x\in J^\perp$, $I(x)J=\{0\}$ and so for any $Q\in D^{\gr}(x)$, $I(x)\subseteq Q$ or $J\subseteq Q$. Since $Q\in D^{\gr}(x)$, the first one cannot happen. Hence, $J\subseteq Q$ showing that $Q\in V^{\gr}(J)$. Therefore, $D^{\gr}(x)$ is an open neighborhood of $P$ contained in $V^{\gr}(J)$. So $P\in \Int(V^{\gr}(J))$, and our claim is established.

Now following exactly the same argument as in the proof of Theorem \ref{th regular ideals via the prime spectrum}, we get $V^{\gr}(I^{\perp \perp})=\overline{\Int(V^{\gr}(I))}$. If $I$ is regular, then from this equality it is clear that $V^{\gr}(I)$ is a regular closed set and hence $D^{\gr}(I)$ is a regular open set. Conversely if $D^{\gr}(I)$ is regular open, then $V^{\gr}(I)$ is regular closed and hence, 
\begin{equation}\label{eq theveq}
V^{\gr}(I)=\overline{\Int(V^{\gr}(I))}=V^{\gr}(I^{\perp\perp}).    
\end{equation}
Since both $I$ and $I^{\perp\perp}$ are graded, they are graded semiprime by \cite[Theorem 6.5(3)]{AHLS}. Arguing as before, it follows that $I=\displaystyle{\bigcap_{P\in V^{\gr}(I)}}P$ and $I^{\perp\perp}=\displaystyle{\bigcap_{P\in V^{\gr}(I^{\perp\perp})}}P$. Then $I=I^{\perp\perp}$ by using (\ref{eq theveq}). The remaining part can also be proved using the fact that any graded ideal of $\KP(\Lambda)$ is graded semiprime.
\end{proof}

With the above theorem in hand, we are now ready to state the following interesting theorem giving further evidence that the talented monoid effectively encodes the algebraic structure of the Kumjian--Pask algebra.

\begin{thm}\label{th lattice isomorphism of regular ideals via TM}
Let $\Lambda$ be a row-finite $k$-graph with no sources. Then the Boolean algebras $\reg(T_\Lambda)$ and $\reg^{\gr}(\KP(\Lambda))$ are isomorphic via the isomorphism sending $J$ in $\reg(T_\Lambda)$ to $I(H_J)$ in $\reg^{\gr}(\KP(\Lambda))$. Consequently, if $\Lambda$ and $\Omega$ are row-finite $k$-graphs without sources such that $T_\Lambda\cong T_\Omega$ as $\mathbb{Z}^k$-monoids, then $\reg^{\gr}(\KP(\Lambda))$ is isomorphic to $\reg^{\gr}(\KP(\Omega))$. 
\end{thm}
\begin{proof}
By Theorem \ref{th the three term connection for KPA}, the spaces $\Spec_{\mathbb{Z}^k}(T_\Lambda)$ and $\Spec^{\gr}(\KP(\Lambda))$ are homeomorphic. This induces a lattice isomorphism between the Boolean algebras $\mathcal{R}\mathcal{O}(\Spec_{\mathbb{Z}^k}(T_\Lambda))$ and $\mathcal{R}\mathcal{O}(\Spec^{\gr}(\KP(\Lambda)))$ which maps the regular open set $D(J)$ to the regular open set $D^{\gr}(I(H_J))$. Now by Theorems \ref{th regular ideals via the prime spectrum} and \ref{th regular ideal and spectrum of KPA}, it follows that $\reg(T_\Lambda)$ is isomorphic to $\reg^{\gr}(\KP(\Lambda))$. The isomorphism indeed maps $J$ to $I(H_J)$. The scheme of the proof is depicted in the following diagram:

\[
\begin{tikzcd}[row sep=huge, column sep=huge]
\text{Spec}_{\mathbb{Z}^k}(T_{\Lambda}) \arrow[r, "\cong", "\text{Theorem \ref{th the three term connection for KPA}}"'] \arrow[dr, phantom, "\bigg\Downarrow"] & \text{Spec}^{\text{gr}}(\KP(\Lambda)) \\
\mathcal{RO}\big(\text{Spec}_{\mathbb{Z}^k}(T_{\Lambda})\big) \arrow[r, "\cong"] \arrow[d, "\text{Theorem \ref{th regular ideals via the prime spectrum}}"', "\cong"] & \mathcal{RO}\big(\text{Spec}^{\text{gr}}(\KP(\Lambda))\big) \arrow[d, "\cong"', "\text{Theorem \ref{th regular ideal and spectrum of KPA}}"] \\
\text{Reg}(T_{\Lambda}) \arrow[r, "\cong"] & \text{Reg}^{\text{gr}}(\KP(\Lambda))
\end{tikzcd}
\]
Since any $\mathbb{Z}^k$-monoid isomorphism preserves and reflects the algebraic preorder, it naturally induces a lattice isomorphism between the Boolean algebras of regular $\mathbb{Z}^k$-order ideals. Hence, the final statement follows from the first part.
\end{proof}
Suppose $\reg\Spec^{\gr}(\KP(\Lambda))$ denotes the subspace of $\Spec^{\gr}(\KP(\Lambda))$ consisting of the regular graded prime ideals of $\KP(\Lambda)$. In view of the first part of Theorem \ref{th regular ideals via the prime spectrum}, it follows that the homeomorphism between $\Spec_{\mathbb{Z}^k}(T_\Lambda)$ and $\Spec^{\gr}(\KP(\Lambda))$ preserves the regular primes, i.e., $\reg\Spec_{\mathbb{Z}^k}$ is homeomorphic to $\reg\Spec^{\gr}(\KP(\Lambda))$. Again by Corollary \ref{cor regular prime order ideals}, $\reg\Spec_{\mathbb{Z}^k}(T_\Lambda)$ is homeomorphic to $\mathscr{P}\mathscr{F}il_r(T_\Lambda)$. We now wish to find the remaining piece of the puzzle by finding the combinatorial counterpart of these homeomorphic spaces. 

Let us fix the notations. Let $\Lambda$ be a row-finite $k$-graph without sources and $H$ a hereditary saturated subset of $\Lambda^0$. We define the \emph{reachability set} of $H$ by 
\begin{equation}\label{eq reachability}
\mathsf{R}(H):=r(s^{-1}(H)).
\end{equation}
The reader may note that this set is denoted by $\overline{H}$ in \cite{Schenkel}, but we prefer the notation $\mathsf{R}(H)$ since we use \textsf{bar} to denote the hereditary saturated closure. The set 
\begin{equation}\label{eq inaccessibility}
H^\perp:=\Lambda^0\setminus \mathsf{R}(H)
\end{equation}
is called the \emph{inaccessibility set} of $H$. For any vertex $w\in \Lambda^0$, we write $T(w):=s(r^{-1}(w))$. 

\begin{prop}\label{pro reachability and double perp}
For any row-finite $k$-graph $\Lambda$ without sources and a hereditary saturated subset $H$ of $\Lambda^0$, $H^\perp$ is a hereditary saturated subset of $\Lambda^0$ and $H^{\perp\perp}=\{w\in \Lambda^0~|~T(w)\subseteq \mathsf{R}(H)\}$.  
\end{prop}
\begin{proof}
Let $r(\lambda)\in H^\perp$ for some $\lambda\in \Lambda$. If $s(\lambda)\in \mathsf{R}(H)$, then we would have $r(\lambda)\in \mathsf{R}(H)$. This forces $s(\lambda)$ to be in $H^\perp$. Hence $H^\perp$ is hereditary. Now assume that $s(v\Lambda^\N)\subseteq H^\perp$ for some vertex $v$ and $\N\in \mathbb{N}^k$. If possible suppose $v\in \mathsf{R}(H)$. Then there exists $\lambda\in \Lambda$ such that $v=r(\lambda)$ and $s(\lambda)\in H$. Choose any $\mu\in s(\lambda)\Lambda^\N$ and let $\tau=\lambda\mu$. Then $\tau(0,\N)\in v\Lambda^\N$ and $s(\tau(\N,d(\tau)))=s(\mu)\in H$ since $r(\mu)=s(\lambda)\in H$ and $H$ is hereditary. But then it follows that $s(\tau(0,\N))\in \mathsf{R}(H)$, contradicting $s(v\Lambda^\N)\subseteq H^\perp$. Therefore, $v\in H^\perp$. This shows that $H^\perp$ is saturated. The proof of the remaining part is exactly the same as that of \cite[Proposition 3.12]{CGHaz}.
\end{proof}
Now consider a regular graded prime ideal $I$ of $\KP(\Lambda)$. The primeness implies $M:=\Lambda^0\setminus H(I)$ is a maximal tail of $\Lambda$. The regularity implies that \[H(I)=\{w\in \Lambda^0~|~T(w)\subseteq \mathsf{R}(H(I))\}\] by \cite[Theorem 4.6(iii)]{Schenkel}. This, together with Proposition \ref{pro reachability and double perp}, shows that $H(I)=H(I)^{\perp\perp}$. Then \[\mathsf{R}\left((\Lambda^0\setminus M)^\perp\right)=\mathsf{R}(H(I)^\perp)=\Lambda^0\setminus H(I)^{\perp\perp}=\Lambda^0\setminus H(I)=M.\] Now if $N$ is any maximal tail with the property that $N=\mathsf{R}\left((\Lambda^0\setminus N)^\perp\right)$, then $J=I(\Lambda^0\setminus N)$ is a graded prime ideal and since \[H(J)^{\perp\perp}=\Lambda^0\setminus \mathsf{R}\left(H(J)^\perp\right)=\Lambda^0\setminus \mathsf{R}\left((\Lambda^0\setminus N)^\perp\right)=\Lambda^0\setminus N=H(J),\] $J$ is regular by \cite[Theorem 4.6(iii)]{Schenkel}.

Define \[\mathscr{M}_r(\Lambda):=\{M\in \mathscr{M}(\Lambda)~|~M=\mathsf{R}\left((\Lambda^0\setminus M)^\perp\right)\}.\] We can now summarize the above discussion through the following diagram:  
\[
\begin{tikzcd}[row sep=1cm, column sep=0.5cm]
& \mathscr{P}\mathscr{F}il(T_\Lambda) \arrow[rr, "\cong"] 
& & \mathrm{Spec}_{\mathbb{Z}^k}(T_\Lambda) \arrow[dl, "\cong"'] \\
\mathscr{M}(\Lambda) \arrow[ur, "\cong"]  
& & \mathrm{Spec}^{\mathrm{gr}}(\KP(\Lambda)) \arrow[ll, "\cong"'{xshift=10pt}] \\
& \mathscr{P}\mathscr{F}il_r(T_\Lambda) \arrow[rr, dashed, "\cong"{xshift=-10pt}] \arrow[uu, hook, dashed] 
& & \mathrm{RegSpec}_{\mathbb{Z}^k}(T_\Lambda) \arrow[uu, hook] \arrow[dl, "\cong"']\\
\mathscr{M}_r(\Lambda) \arrow[uu, hook] \arrow[ur, dashed, "\cong"] 
& & \mathrm{RegSpec}^{\mathrm{gr}}(\KP(\Lambda)) \arrow[ll, "\cong"'] \arrow[uu, hook] 
\end{tikzcd}
\]
where the hook arrows indicate inclusions as subspaces and the homeomorphisms in the bottom row are the restrictions of the corresponding homeomorphisms in the top row.

\subsection{Talented monoid criterion for purely infinite higher-rank graph algebras} In \cite{HLM}, we obtained a talented monoid criterion for purely infinite simple Kumjian--Pask algebras and $C^*$-algebras of row-finite higher-rank graphs without sources. Using the talented monoid description of maximal tails obtained in Section \ref{sec maximal tail and TM}, we now obtain a sufficient criterion for purely infinite (not necessarily simple) higher-rank graph algebras. 

\begin{thm}\label{th sufficient criterion for purely infinite higher-rank algebras}
Let $\Lambda$ be a row-finite $k$-graph with no sources. Then the following conditions are equivalent for the talented monoid $T_\Lambda$.

$(i)$ Every prime $\mathbb{Z}^k$-order ideal $P$ of $T_\Lambda$ is primitive, and for any $0\neq a\in S=T_\Lambda/{P}$, there exist $x\in S$ and $\N\in \mathbb{N}^k\setminus \{0\}$ such that $^{\N} x< x \leq a$.

$(ii)$ Every prime $\mathbb{Z}^k$-order ideal $P$ of $T_\Lambda$ is primitive, and for any nonzero order ideal $J$ of $S=T_\Lambda/{P}$, there exist $x\in J$ and $\N\in \mathbb{N}^k\setminus \{0\}$ such that $^\N x< x$. 

$(iii)$ For any proper $\mathbb{Z}^k$-order ideal $I$ of $T_\Lambda$, $\mathbb{Z}^k$ acts freely on the quotient $T_\Lambda/{I}$, and for every prime $\mathbb{Z}^k$-order ideal $P$ of $T_\Lambda$ and any nonzero order ideal $J$ of $S=T_\Lambda/{P}$, there exist $x\in J$ and $\N\in \mathbb{N}^k\setminus \{0\}$ such that $^\N x< x$.

If any one of the above holds, then both $C^*(\Lambda)$ and $\KP(\Lambda)$ are purely infinite. 
\end{thm}
\begin{proof}
$(i)\Rightarrow (ii)$ Let $P$ be any prime $\mathbb{Z}^k$-order ideal of $T_\Lambda$ and $S=T_\Lambda/{P}$. We only need to prove the second part of $(ii)$. So assume that $J$ is any nonzero order ideal of $S$. Choose any $0\neq a\in J$. Then by $(i)$, there exist $x\in S$ and $\N\in \mathbb{N}^k\setminus \{0\}$ such that $a\ge x>~^\N x$. Since $J$ is an order ideal, $x\in J$ and we are done.

$(ii)\Leftrightarrow (iii)$ Since $T_\Lambda$ is a refinement monoid, this is evident from Proposition \ref{pro free action on prime quotient is enough}. 

$(iii)\Rightarrow (i)$ Let $P$ be any prime $\mathbb{Z}^k$-order ideal of $T_\Lambda$. Since $\mathbb{Z}^k$ acts freely on $T_\Lambda/{P}$, $P$ is primitive. Now let $0\neq a\in S=T_\Lambda/{P}$. Let $J:=\{x\in S~|~x\le a\}$. Then $J$ is a nonzero order ideal of $S$. By the condition, there is an $x\in J$ and $\N\in \mathbb{N}^k\setminus \{0\}$ such that $^\N x<x$. Since $x\in J$, $x\le a$ and we are done.

Now assume that $(i)$ holds. Let $M$ be any maximal tail in $\Lambda$ and $(\mu,\nu)$ a generalized cycle in $M$. Let $P:=J_{\Lambda^0\setminus M}$. Then $P$ is a prime $\mathbb{Z}^k$-order ideal of $T_\Lambda$ and $S=T_\Lambda/{P}\cong T_{\Lambda/(\Lambda^0\setminus M)}$. Since $\mathbb{Z}^k$ acts freely on $S$, by \cite[Proposition 4.2.6$(ii)$]{HLM}, the generalized cycle $(\mu,\nu)$ $\Lambda/(\Lambda^0\setminus M)$ has an entrance $\tau$. Then $s(\tau)\in M$. Therefore $(\mu,\nu)$ has an entrance in $M$. Now suppose $v\in M$. Then $0\neq v(0)\in S=T_\Lambda/{P}$. By $(iii)$, there exists $x\in S$ and $\N\in \mathbb{N}^k\setminus \{0\}$ such that $v(0)\ge x>~^\N x$ in $S$. Applying \cite[Lemma 4.2.1]{HLM}, we can conclude that $v$ is reachable from a cycle (and hence generalized cycle) in $\Lambda/(\Lambda^0\setminus M)$ or equivalently, from a generalized cycle in $M$. Finally, by applying \cite[Theorem 3.9]{PaskNew}, it follows that $C^*(\Lambda)$ is purely infinite in the sense of Kirchberg--R\o rdam. 

For the Kumjian--Pask algebra case, by \cite[Theorems 3.4 and 4.2]{LPIS2}, it suffices to show that every nonzero one-sided ideal of every quotient of $\KP(\Lambda)$ contains an infinite idempotent. The first part of $(iii)$, together with \cite[Theorem 5.11]{HMPS1}, implies that $\Lambda$ is strongly aperiodic. Hence by \cite[Theorem 5.6]{Pino}, every ideal of $\KP(\Lambda)$ is graded and is of the form $I(H)$ for some hereditary saturated subset $H$ of $\Lambda^0$. So take any proper hereditary saturated subset $H$ and consider the quotient $\KP(\Lambda)/I(H)\cong \KP(\Lambda/H)$. We denote the Kumjian--Pask family in $\KP(\Lambda/H)$ by $(P,S)$. Let $I$ be any nonzero right ideal of $\KP(\Lambda/H)$. Our aim is to show that $I$ contains an infinite idempotent of $\KP(\Lambda/H)$. The case for left ideals will follow mutatis mutandis. Let $0\neq x\in I$. Since $\Lambda$ is strongly aperiodic, $\Lambda/H$ is aperiodic, and so we can apply \cite[Proposition 4.9]{Pino} to obtain $a,b\in \KP(\Lambda/H)$ such that $a xb=P_w$ for some $w\in \Lambda^0\setminus H$. Now $\KP(\Lambda/H)$ is Jacobson semisimple (which can be proved easily by using \cite[Proposition 4.9]{Pino}), and so there must exist a primitive ideal $P$ of $\KP(\Lambda/H)$ such that $P_w\notin P$, as otherwise $P_w$ would be in the intersection of all primitive ideals, which is nothing but the Jacobson radical. Since $P$ is primitive, it is prime. By \cite[Proposition 4.1]{Larki}, it is of the form $P=I\left(((\Lambda^0\setminus H)\setminus M\right)$ for some maximal tail $M$ in $\Lambda/H$. Now since $H$ is hereditary, $M$ is also a maximal tail in $\Lambda$. Since $P_w\notin P$, $w\in M$. By $(ii)$, there exists a generalized cycle $(\mu,\nu)$ in $M$ such that $w(\Lambda/H)s(\mu)\neq \emptyset$. Let $\lambda\in w(\Lambda/H)s(\mu)$. From the preceding paragraph, it follows that the generalized cycle $(\mu,\nu)$ has an entrance in $M$. Then the proof of \cite[Proposition 4.2.7]{HLM} implies that $S_\nu S_{\nu^*}$ is an infinite idempotent in $\KP(\Lambda/H)$. Consequently, $P_{s(\nu)}=S_{\nu^*}S_\nu$ is an infinite idempotent in $\KP(\Lambda/H)$. Now \[P_{s(\nu)}=S_{\lambda^*}P_wS_\lambda=S_{\lambda^*} a x b S_\lambda.\] Let $g=xbS_{\lambda}S_{\lambda^*}a$. Then $g^2=g$ and $g\in I$ since $I$ is a right ideal and $x\in I$. Moreover, since \[(S_{\lambda^*}a)g(xbS_{\lambda})=P_{s(\nu)}P_{s(\nu)}=P_{s(\nu)},\] we have $P_{s(\nu)}\precsim g$. Since Kumjian--Pask algebras are $s$-unital, by \cite[Lemma 3.9$(iii)$]{AGPM}, $g$ is an infinite idempotent of $\KP(\Lambda/H)$ which is in $I$. 
\end{proof}

When $\Lambda=E^*$ for a row-finite directed graph $E$, the talented monoid conditions stated in Theorem \ref{th sufficient criterion for purely infinite higher-rank algebras} are also necessary for pure infiniteness of the Kumjian--Pask algebra of $\Lambda$.

\begin{prop}\label{prop Purely infinite LPA via TM}
Let $E$ be a row-finite directed graph. Then $L(E)$ is purely infinite if and only if every prime $\mathbb Z$-order ideal $P$ of $T_E$ is primitive, and for any $0\neq a\in S=T_E/{P}$, there is an $x\in S$ and $i\in \mathbb N\setminus \{0\}$ such that $^i x< x \leq a$.
\end{prop}
\begin{proof}
Suppose the condition holds in $T_E$. Note that the arguments used in the proof of Theorem \ref{th sufficient criterion for purely infinite higher-rank algebras} for Kumjian--Pask algebras can also be applied for $L(E)$ to show that every nonzero left (right) ideal of every quotient of $L(E)$ contains an infinite idempotent. Therefore, $L(E)$ is purely infinite in view of \cite[Theorem 7.4]{AGPM}. 

Conversely, assume that $L(E)$ is purely infinite. Suppose $P$ is a prime $\mathbb Z$-order ideal of $T_E$. Then $M:=E^0\setminus H_P$ is a maximal tail in $\Lambda$ and $S=T_E/P\cong T_{E/H_P}$. By \cite[Theorem 7.4]{AGPM}, every cycle in $M$ has an exit in $M$. Since $(E/H_P)^0=M$, it follows that the quotient graph $E/H_P$ satisfies Condition (L). Then by \cite[Corollary 4.3$(i)$]{hazli}, $\mathbb Z$ acts freely on $S$. Hence, $P$ is primitive. Let $0\neq a\in S$. In view of the isomorphism $S\cong T_{E/H_P}$, there exists a vertex $v\in M$ and $j\in \mathbb N$ such that $a\ge v(j)$. Now $v$ connects to a cycle $c$ in $M$. Let $l=|c|$. Then $v\ge s(c)\ge s(c)(l)$. But since $\mathbb Z$ acts freely on $S$ and $s(c)(0)\in S$, we must have $s(c)> s(c)(l)$. Hence, $a\ge s(c)(j)>~^ls(c)(j)$. This completes the proof.
\end{proof}

\textbf{AI disclosure statement:} The authors acknowledge the use of {\bf ChatGPT (GPT-5.6 Luna)} exclusively for proofreading, correcting typographical errors, and to draw the diagrams shown in this paper. All the core mathematical concepts, insights, and the structural arguments used in the proofs are the authors' original work and they take full responsibility for these.

\textbf{Acknowledgments:} This work was initiated at the Centre for Research in Mathematics and Data Science, Western Sydney University while the second author was visiting the first author in July--August 2026. Mukherjee wholeheartedly thanks Hazrat for his warm hospitality and all the support during the visit. Hazrat acknowledges Australian Research Council Discovery Project DP230103184.

\end{document}